\documentclass[12pt,letterpaper]{amsart}
\usepackage[a4paper,top=3cm,bottom=3cm,inner=2.5cm,outer=2.5cm]{geometry}
\usepackage{amsmath,amsfonts,amssymb,amsthm}
\usepackage{xparse}
\usepackage[all]{xy}
\usepackage{tikz}
\usepackage{tikz-cd}
\usepackage{multirow,enumitem,array}
\usepackage{pst-node,placeins,xspace,fix-cm}
\usepackage[nomessages]{fp}
\usepackage[colorlinks,linkcolor=blue,citecolor=violet,urlcolor=darkgray,final,hyperindex,linktoc=page,hyperfootnotes=true]{hyperref}

\newcolumntype{F}{>{$}c<{\hspace{-0.9ex}$}}
\newcolumntype{:}{>{$}m{0.8ex}<{$}}
\newcolumntype{R}{>{$}r<{$}}
\newcolumntype{C}{>{$}c<{$}}
\newcolumntype{L}{>{$}l<{$}}
\newcolumntype{N}{@{}>{$}l<{$}}
\newlength\horspace
\newcommand{\h}[1][1.0]{\hspace*{#1\horspace}}
\newlength\verspace
\renewcommand{\v}[1][1.0]{\vspace*{#1\verspace}\xspace}
\tikzset{iso/.style={draw=none,every to/.append style={edge node={node [sloped, allow upside down, auto=false]{$\cong$}}}}}
\tikzset{adjunction/.style={draw=none,every to/.append style={edge node={node [sloped, allow upside down, auto=false]{$\dashv$}}}}}
\tikzset{simeq/.style={draw=none,every to/.append style={edge node={node [sloped, allow upside down, auto=false]{$\simeq$}}}}}
\tikzset{simeqS/.style={draw=none,every to/.append style={edge node={node [sloped, allow upside down, auto=false]{$\raisebox{0.8em}{$\simeq$}$}}}}}
\tikzset{aiso/.style={simeqS,preaction={draw,->}}}
\tikzset{proarrowS/.style={draw=none,every to/.append style={edge node={node [sloped, allow upside down, auto=false]{\raisebox{1.4pt}{\small$\shortmid$}}}}}}
\tikzset{proarrow/.style={proarrowS,preaction={draw,->}}}
\tikzset{dotdot/.style={dash pattern=on 0.25ex off 0.2ex, dash phase=0ex}}
\tikzset{RightA/.style={double distance=3.5pt,>={Implies},->},%
	triple/.style={-,preaction={draw,RightA}},%
	quadruple/.style={preaction={draw,RightA,shorten >=0pt},shorten >=1pt,-,double,double distance=0.2pt}}
\tikzset{Right/.style={double distance=1.7pt,>={Implies},->}}
\tikzset{simeqSRight/.style={draw=none,every to/.append style={edge node={node [sloped, allow upside down, auto=false]{$\raisebox{-1em}{\rotatebox{180}{$\simeq$}}$}}}}}
\tikzset{twoiso/.style={simeqSRight,preaction={draw,Right}}}

\theoremstyle{plain} % default style
\newtheorem{theorem}{Theorem}[section]
\newtheorem{lemma}[theorem]{Lemma}
\newtheorem{proposition}[theorem]{Proposition}
\newtheorem{corollary}[theorem]{Corollary}

\theoremstyle{definition}
\newtheorem{definition}[theorem]{Definition}
\newtheorem{remark}[theorem]{Remark}
\newtheorem{example}[theorem]{Example}
\newtheorem{construction}[theorem]{Construction}

\def\nameit#1{\textrm{#1}~}
\def\thex{\nameit{Theorem}}
\def\prox{\nameit{Proposition}}
\def\corx{\nameit{Corollary}}
\def\lemx{\nameit{Lemma}}
\def\defx{\nameit{Definition}}
\def\remx{\nameit{Remark}}

\def\conx{\nameit{Construction}}
\def\dfn#1{{\itshape #1}}
\def\predfn#1{{\itshape #1}}
\newcommand{\refs}[1]{\textup{(}\ref{#1}\textup{)}}

\NewDocumentEnvironment{cd}{s O{6} O{6} b}{%
	\IfBooleanF{#1}{\begin{equation*}}\begin{tikzcd}[row sep=#2ex,column sep=#3ex,ampersand replacement=\&]
			#4
		\end{tikzcd}\IfBooleanF{#1}{\end{equation*}}\ignorespacesafterend}{}

\newenvironment{fun}{\[\begin{tabular}{F:RCL}}{\end{tabular}\]\ignorespacesafterend}

\newenvironment{eqD*}{\begin{equation*}}{\end{equation*}\ignorespacesafterend}

\newcommand{\C}{\mathbb{C}}
\newcommand{\D}{\mathbb{D}}
\newcommand{\T}{\mathcal{T}}
\newcommand{\F}{\mathcal{F}}
\newcommand{\0}{\mathbf{0}}
\renewcommand{\1}{\mathbf{1}}
\newcommand{\I}{\mathcal{I}}
\newcommand{\J}{\mathcal{J}}
\newcommand{\PTors}{\mathbf{PTor}}
\newcommand{\CartPTors}{\mathbf{PTor}_{\operatorname{R}}}
\newcommand{\SCartPTors}{\mathbf{PTor}_{\operatorname{SR}}}
\newcommand{\CartTors}{\mathbf{Tor}_{\operatorname{R}}}
\newcommand{\Pointcat}{\mathbf{PointCat}}
\newcommand{\Biqpointcat}{\mathbf{BiQPointCat}}
\NewDocumentCommand{\Alg}{t+ t' m}{
	\ensuremath{\IfBooleanT{#1}{\mathbf{Ps}\mbox{-}}{#3}\mbox{-}\mathbf{\IfBooleanT{#2}{Co}Alg}}
}

\def\:{\colon}

\def\c{\circ}
\newcommand{\iso}{\cong}
\def\phi{\varphi}

\newcommand{\unit}{\operatorname{unit}}
\newcommand{\counit}{\operatorname{counit}}

\newcommand{\fami}[2]{\left\{{#1}\right\}_{#2}}
\newcommand{\cont}{\subseteq}

\newcommand{\id}[1]{\operatorname{id}_{#1}}
\newcommand{\Id}[1]{\operatorname{Id}_{#1}}

\newcommand{\x}[1][]{\times_{#1}}
\newcommand{\too}{\longrightarrow}
\newcommand{\mto}{\mapsto}
\newcommand{\aar}[2][]{\xrightarrow[#1]{#2}}
\makeatletter
\newcommand{\aR}[2][]{%
	\ext@arrow 0055{\Rightarrowfill@}{#1}{#2}}
\def\xLongrightarrowfill@{\arrowfill@\Relbar\Relbar\Longrightarrow}
\newcommand{\am}[2][]{%
	\ext@arrow 0395\xmapstofill@{#1}{#2}}
\def\xlongmapstofill@{\arrowfill@\relbar\relbar\longmapsto}
\def\xlongrightarrowfill@{\arrowfill@\relbar\relbar\longrightarrow}
\newcommand{\aarr}[2][]{%
	\ext@arrow 0099\xlongrightarrowfill@{#1}{#2}}
\newcommand{\eqq}{\DOTSB\protect\Relbar\protect\joinrel\Relbar}
\def\xeqqfill@{\arrowfill@\Relbar\Relbar\eqq}
\newcommand{\aeqq}[2][]{%
	\ext@arrow 0099\xeqqfill@{#1}{#2}}
\makeatother
\newcommand{\aiso}[1]{\overset{#1}{\iso}}

\newcommand{\PB}[1]{\arrow[#1,phantom,"\scalebox{1.6}{\color{black}$\lrcorner$}",very near start]}
\newcommand{\scaleu}[2][1.2]{{\scalebox{#1}{$#2$}}}
\NewDocumentCommand{\fib}{O{n} O{2.3} mmm}{%
	\begin{cd}*[#2][5]
		{#3}\ifx#1n{\arrow[d,"{\,\scaleu{#4}}"]}\else{\ifx#1i{\arrow[d,hookrightarrow,"{\,\scaleu{#4}}"]}\else{\ifx#1e{\arrow[d,equal,"{\,\scaleu{#4}}"]}\else{\ifx#1R{\arrow[d,Rightarrow,"{\,\scaleu{#4}}"]}\fi}\fi}\fi}\fi\\
		{#5}\ifx#1o{\arrow[u,"{\,\scaleu{#4}}"']}\fi
	\end{cd}\xspace
}
\newcommand{\Ar}[4][]{\arrow[#2,"{#3}"{#1},""{name=#4, anchor=center}]}
\newcommand{\Ars}[4][]{\arrow[#2,"{#3}"'{#1},""{name=#4, anchor=center}]}
\newcommand{\Arb}[6][]{\arrow[#2,"{#3}"{#1},from=#4,to=#5,shorten <= #6 em, shorten >= #6 em]}

\NewDocumentCommand{\tcv}{s t' O{5} O{38} mmmmm}{
	\FPmul\Mulresulttwo{#3}{#3}%
	\FPmul\Mulresult{0.0026}{\Mulresulttwo}%
	\IfBooleanTF{#1}{\begin{cd}*}{\begin{cd}}[#3][#3]
			{#5}\IfBooleanTF{#2}{\Ars{d,leftarrow,bend right=#4}{#7}{A}\Ar{d,leftarrow,bend left=#4}{#8}{B}}{\Ars{d,bend right=#4}{#7}{A}\Ar{d,bend left=#4}{#8}{B}}\\{#6}
			\Arb{Rightarrow}{#9}{A}{B}{\Mulresult}
		\end{cd}}
\NewDocumentCommand{\sq}{s O{n} O{6} O{6} O{} O{2.7} O{2.2} O{0.5} O{n}}{%
	\def\foosq##1##2##3##4##5##6##7##8{%
		\IfBooleanTF{#1}{\begin{cd}*}{\begin{cd}}[#3][#4]
				{##1}\ifx#2p{\PB{rd}}\fi\arrow[r,"{##5}"]\ifx#9l{\arrow[d,equal,"{##6}"']}\else{\arrow[d,"{##6}"']}\fi\&{##2}\ifx#9r{\arrow[d,equal,"{##7}"]}\else{\arrow[d,"{##7}"]}\fi\ifx#2l{\arrow[ld,Rightarrow,shorten <=#6ex,shorten >=#7ex,"{#5}"{pos=#8}]}\fi \ifx#2i{\arrow[ld,twoiso,shorten <=#6ex,shorten >=#7ex,"{#5}"{pos=#8}]}\fi\\
				{##3}\ifx#9d{\arrow[r,equal,"{##8}"']}\else{\arrow[r,"{##8}"']}\fi\ifx#2o{\arrow[ur,Rightarrow,shorten <=#6ex,shorten >=#7ex,"{#5}"{pos=#8}]}\fi\&{##4}
		\end{cd}}%
		\foosq}

\begin{document}

\title{Rectangular torsion theories}

\author{Elena Caviglia}
\address{(Elena Caviglia) [1] Department of Mathematical Sciences, Stellenbosch University, South Africa. [2]  National Institute for Theoretical and Computational Sciences (NITheCS), Stellenbosch, South Africa.}
\email{elena.caviglia@outlook.com}

\author{Zurab Janelidze}
\address{(Zurab Janelidze) [1] Department of Mathematical Sciences, Stellenbosch University, South Africa. [2]  National Institute for Theoretical and Computational Sciences (NITheCS), Stellenbosch, South Africa.}
\email{zurab@sun.ac.za}

\author{Luca Mesiti}
\address{(Luca Mesiti) Department of Mathematics, University of KwaZulu-Natal, South Africa.}
\email{luca.mesiti@outlook.com}


\begin{abstract} 
In this paper we introduce and study \emph{rectangular torsion theories}, i.e.\ those torsion theories $(\C,\T,\F)$ with $\C$ a pointed category, 
%in the sense of A.~Facchini, C.~Finocchiaro, and M.~Gran 
where the canonical functor $\C\to \mathcal{T}\times\mathcal{F}$ is an equivalence of categories. In particular, we show that these are precisely the internal rectangular bands in the $2$-category of pointed categories.

\end{abstract}

\maketitle

\setcounter{tocdepth}{1}
\tableofcontents

\section{Introduction}

A \predfn{pretorsion theory} in the sense of \cite{FFG21} is a triple $(\C,\T,\F)$, where $\C$ is a category and $(\T,\F)$ is a pair of full replete subcategories of $\mathcal{C}$, such that for the ideal of morphisms that factor through objects in the intersection $\mathcal{Z}=\mathcal{T}\cap\mathcal{F}$, the following conditions hold:
\begin{itemize}
    \item[(T1)] Every morphism in $\C$ from an object in $\mathcal{T}$ to an object in $\mathcal{F}$ is a \emph{null morphism}, i.e., factors through an object in $\mathcal{Z}$.

    \item[(T2)] Every object $X$ of $\mathcal{C}$ sits in the middle of a short-exact sequence from an object in $\T$ to an object in $\F$, i.e., for every $X$ there is a sequence of morphisms
    $$T\aar{\ell^X} X\aar{r^X} F$$
    such that $T\in \T$, $F\in\F$, $\ell^X$ is a kernel of $r$ and $r^X$ is a cokernel of $\ell$. 
\end{itemize}
Here kernels and cokernels are relative to the ideal $\mathcal{N}$ of morphisms that factor through an object in $\mathcal{Z}$. These are by now well-established notions (see \cite{FFG21} and the references there), so we do not recall them here. 

The notion of a pretorsion theory is a vast generalisation of the classical notion of torsion theory in an abelian category introduced in \cite{Dic66}. There are various intermediate generalisations in the literature --- see \cite{FFG21,GJ20} and the references there. In \cite{FFG21}, it is shown that the notion of a pretorsion theory is highly versatile in the sense that, on one hand, many properties can be deduced from the simple definition, and on the other hand, there is a large variety of examples. Of particular interest to us is the special case of a pretorsion theory where $\mathcal{Z}$ is the class of zero objects in a pointed category --- we call such pretorsion theory a \emph{torsion theory} in this paper (following \cite{JT07}, where the pointed category is additionally assumed to have all kernels and cokernels).

In this paper we explore a special class of (pre)torsion theories that are analogous, in a precise sense to be explained below, to \emph{rectangular bands} in semigroup theory \cite{CliPre1964}, i.e., idempotent semigroups satisfying $xyz=xz$. It is not difficult to show that rectangular bands are algebras for a monad on $\mathbf{Set}$ given by $X\mapsto X\times X$. In fact, the proof of this is Yoneda invariant, and so internal rectangular bands in any category with squares (i.e., where $X\times X$ exists for each object $X$), are the same as algebras over a similar monad. This suggests to define internal rectangular bands in a $2$-category (having squares of objects) as pseudo-algebras over the corresponding $2$-monad.

We show (Theorem~\ref{teormonadrestricted}) that there is a $2$-equivalence between the $2$-category of internal rectangular bands in the (large) $2$-category of pointed categories and the $2$-category of torsion theories $(\C,\T,\F)$ where the canonical functor $\C\to\T\times\F$ is an equivalence of categories. We call such torsion theories \emph{rectangular torsion theories}. The proof leads to yet another description of rectangular torsion theories (Theorem~\ref{theorcharactcartpointed}): they are precisely those torsion theories that are equivalent to products, in the $2$-category of torsion theories, of torsion theories of the form $(\T,\T,\0)$ and of the form $(\F,\0,\F)$, where in each case $\0$ is the full subcategory of the base category consisting of all zero objects. This is analogous to the well known fact about rectangular bands that they are precisely the products of left-zero bands (those given by first projection, i.e., $xy=x$) and right-zero bands ($xy=y$). A closely related fact is that the category of non-empty rectangular bands is equivalent to the rectangular square of the category of non-empty sets. An analogous result can be established in our case as well: the $2$-category of rectangular torsion theories is equivalent to the rectangular square of the $2$-category of pointed categories. 

Thus, up to an equivalence of categories, a rectangular torsion theory is given by a canonical torsion theory on the rectangular product $\mathbb{C}\times \mathbb{D}$ of two pointed categories. This torsion theory is as natural as it is easy to guess what it should be: $\mathcal{T}=\mathbb{C}\times\mathbf{0}_\mathbb{D}$ and $\mathcal{F}=\mathbf{0}_{\mathbb{C}}\times\mathbb{D}$, where $\mathbf{0}_\mathbb{X}$ denotes the full subcategory of $\mathbb{X}$ consisting of zero objects (one could also swap the roles of $\mathcal{T}$ and $\mathcal{F}$). For this torsion theory, we have the following.
\begin{itemize}
\item The axiom (T1) says that every morphism $(C,0)\to (0,D)$ is a null morphism.

\item The short-exact sequences for (T2) can be given by
$$(C,0)\aar{(\id{C},!)}(C,D)\aar{(!,\id{D})} (0,D).$$

\item Rectangularity is expressed by the obvious equivalence $$\mathbb{C}\times\mathbb{D}\approx (\mathbb{C}\times \mathbf{0}_{\mathbb{D}})\times (\mathbf{0}_{\mathbb{C}}\times\mathbb{D}).$$
\end{itemize}
We therefore see that similarly to rectangular bands, rectangular torsion theories are very easy to construct! Actually, as it turns out, it is not the construction, but the detection of a rectangular torsion theory that proves to be really interesting, as we demonstrate in the case of the following one particular type of torsion theories.

Consider a pointed category $\mathbb{X}$ having finite limits and a class $\mathcal{E}$ of epimorphisms in $\mathbb{X}$ that contains all isomorphisms and all morphisms to a zero object (both of these types of morphisms are clearly epimorphisms). Let us view $\mathcal{E}$ as a full subcategory of the arrow category of $\mathbb{C}$. Let $\mathcal{T}$ denote its (full) subcategory consisting of isomorphisms and let $\mathcal{F}$ denote its (full) subcategory consisting of morphisms to a zero object. We establish the following.
\begin{itemize}
\item $(\mathcal{E},\mathcal{T},\mathcal{F})$ is a torsion theory if and only if every element of $\mathcal{E}$ is a normal epimorphism in $\mathbb{X}$. It is a rectangular torsion theory if and only if every element of $\mathcal{E}$ is a normal product projection in $\mathbb{X}$. See Theorem~\ref{TheD}.

\item Consequently, when $\mathcal{E}$ is the class of product projection in $\mathbb{X}$, the triple $(\mathcal{E},\mathcal{T},\mathcal{F})$ is a torsion theory if and only if it is a rectangular torsion theory and if and only if the isomorphism theorem $X\times Y/X\approx Y$ holds in $\mathbb{X}$ (i.e., $\mathbb{X}$ has normal projections in the sense of \cite{Janelidze03}). See Theorem~\ref{TheE}.

\item When $\mathcal{E}$ is the class of regular epimorphisms in a regular category $\mathbb{X}$, the triple $(\mathcal{E},\mathcal{T},\mathcal{F})$ is a torsion theory if and only if $\mathbb{X}$ is a normal category in the sense of \cite{Janelidze10} (see Theorem~\ref{TheF}). When $\mathbb{X}$ is normal, the triple $(\mathcal{E},\mathcal{T},\mathcal{F})$ is a rectangular torsion theory if and only if every regular epimorphism is a product projection (see Theorem~\ref{TheG}).
\end{itemize}
The property of a regular category to have all of its regular epimorphisms product projections seems to be a new property in categorical algebra. Examples of such normal categories are categories of vector spaces over a fixed field, as well as the dual of the category of pointed objects in any topos.

While the final goal of our paper is the study of rectangular torsion theories in the pointed case, since this is where the interesting examples seem to be found, we establish different parts of our theory in different weaker contexts. Some of our results deal with exploration of rectangularity for pretorsion theories in the context of general categories (see Section~\ref{SecA}). Before specializing to pointed categories, the $2$-monadicity result is first established (Theorem~\ref{teormonad}) for categories having an initial object $0$ and a terminal object $1$, with the unique morphism $0\to 1$ being both an epimorphism and a monomorphism (we call such categories \emph{bi-quasi-pointed}, as they are essentially dualized versions of quasi-pointed categories in the sense of D.~Bourn \cite{Bourn01}).

\subsection*{Notations}

We will always assume pretorsion theories to be equipped with a choice of a short exact sequence $T^X\aar{\ell^X} X\aar{r^X} F^X$ for every object $X$ of $\C$, denoted in this way. This also induces choices of morphisms $h^T$ and $h^F$ making the following diagram commute
\begin{cd}[4][6]
    T^X\ar[d,"h^T"']\ar[r,"\ell^X"] \& X \ar[d,"h"]\ar[r,"r^X"] \& F^X\ar[d,"h^F"]\\
    T^Y\ar[r,"\ell^Y"'] \& Y \ar[r,"r^Y"'] \& F^Y
\end{cd}
for every $h\:X\to Y$, by uniqueness of $h^T$ and $h^F$.

Notice that if $n\:X\to Y$ is a null morphism, then $\id{X}$ is a kernel of $n$ and $\id{Y}$ is a cokernel of $n$. So given $T\in \T$, we have that $T^T\iso T$, since $r^T\:T\to F^T$ is a null morphism by (T1).

We will thus always choose $T^T=T$ and $\ell^T=\id{}$ for every $T\in \T$. Analogously, we will always choose $F^F=F$ and $r^F=\id{}$ for every $F\in \F$. In particular, for every $Z\in \T\cap\F$, we will choose $\ell^Z=\id{Z}$ and $r^Z=\id{Z}$.

\section{Rectangular pretorsion theories}\label{SecA}

In this section, we introduce \predfn{rectangular pretorsion theories}. Unlike in the pointed case, we do not have a monadic description of these pretorsion theories. We nevertheless are able to characterize them using products of pretorsion theories. This characterization will serve as a stepping stone for the monadic description of rectangular torsion theories obtained in the next section. 

Let $\C$ be a category having a terminal object, and let $\1$ denote the subcategory of terminal objects in $\C$. Let $\D$ be a category having an initial object, and let $\0$ denote the subcategory on initial objects in $\D$. We explore what it takes for $(\C\x\D,\mathbb{C}\times \mathbf{0},\mathbf{1}\times\mathbb{D})$ to be a pretorsion theory. Note that:
\begin{itemize}
    \item $\mathbb{C}\times \mathbf{0}$ and $\mathbf{1}\times\mathbb{D}$ are each closed under isomorphisms in $\mathbb{C}\times\mathbb{D}$.

    \item The intersection $\mathcal{Z}=(\mathbb{C}\times \mathbf{0})\cap(\mathbf{1}\times\mathbb{D})=\mathbf{1}\times\mathbf{0}$ consists of all $(1,0)$, where $1$ is a terminal object in $\mathbb{C}$ and $0$ is an initial object in $\mathbb{D}$.

    \item Therefore, a morphism $(X,0)\to (1,Y)$ with $(X,0)\in\mathbb{C}\times\mathbf{0}$ and $(1,Y)\in\mathbf{1}\times\mathbb{D}$ is always null (i.e., it factorizes through an object in $\mathcal{Z}$, namely, through $(1,0)$).
\end{itemize}
So, for $(\mathbb{C}\times \mathbf{0},\mathbf{1}\times\mathbb{D})$ to be a pretorsion theory, it is necessary and sufficient that every object $(C,D)$ is part of a short exact sequence
\begin{equation}\label{EquA}\xymatrix@=30pt{(X,0)\ar[r]^-{(m_1,m_2)} & (C,D)\ar[r]^-{(e_1,e_2)} & (1,Y)}\end{equation}
where $0\in\mathbf{0}$ and $1\in\mathbf{1}$.
We note that (\ref{EquA}) is a short exact sequence if and only if the following conditions hold:
\begin{itemize}
\item[(E1)] $(m_1,m_2)$ is a monomorphism. 

\item[(E2)] $(e_1,e_2)$ is an epimorphism. 

\item[(E3)] The composite $(e_1,e_2)\circ (m_1,m_2)$ is null; and indeed it is, as it factors through $(0,1)\in\mathcal{Z}$.

\item[(E4)] Whenever a composite $(e_1,e_2)\circ(u_1,u_2)$ is null, we have $(u_1,u_2)=(m_1,m_2)\circ (u'_1,u'_2)$ for some $u'_1,u'_2$.

\item[(E5)] Whenever a composite $(v_1,v_2)\circ(m_1,m_2)$ is null, we have $(v_1,v_2)=(v'_1,v'_2)\circ (e_1,e_2)$ for some $v'_1,v'_2$.
\end{itemize}

\begin{lemma}\label{LemA}
The diagram (\ref{EquA}) in $\mathbb{C}\times\mathbb{D}$, where $0\in\mathbf{0}$ and $1\in\mathbf{1}$, is a short exact sequence if and only if $m_1$ is an isomorphism, $m_2$ is a monomorphism, $e_1$ is an epimorphism and $e_2$ is an isomorphism.     
\end{lemma}

\begin{proof} First, we make some general observations:
\begin{itemize}
    \item[(G)] $(f_1,f_2)$ is a monomorphism/epimorphism if and only if so are each $f_i$.
\end{itemize}
Now, suppose (\ref{EquA}) is a short exact sequence. (G) together with (E1-2) leave us to show that $m_1$ is a split epimorphism and symmetrically, $e_2$ is a split monomorphism. Thanks to commutativity of the right-hand side square below and (E4), we get a commutative triangle on the left:
$$
\xymatrix@=30pt{(X,0)\ar[r]^-{(m_1,m_2)} & (C,D)\ar[r]^-{(e_1,e_2)} & (1,Y)\\ & (C,0)\ar[ul]^-{(u'_1,\id{0})}\ar[u]_-{(\id{C},m_2)}\ar[r]_-{(e_1,\id{0})} & (1,0)\ar[u]_-{(\id{1},e_2m_2)}} $$
Then $m_1u'_1=\id{C}$, proving that $m_1$ is a split epimorphism. That $e_2$ is a split monomorphism can be proved similarly (with a dual-symmetric argument).

For the converse, suppose all the assumptions stated for $m_i,e_i$ in the lemma hold. It is not difficult to see that since $m_1$ and $e_2$ are isomorphisms, without loss of generality we can assume that they are actually identity morphisms. So the sequence (\ref{EquA}) becomes the top line of the following diagram:
$$
\xymatrix@=30pt{(C,0)\ar[r]^-{(\id{C},m_2)} & (C,D)\ar[r]^-{(e_1,\id{D})} & (1,D)\\ & (U_1,U_2)\ar[u]_-{(u_1,u_2)}\ar[r]_-{(e_1u_1,d)}\ar@{-->}[ul]^-{(u_1,d)} & (1,0)\ar[u]_-{(\id{1},m_2)}} $$
(E1-2) hold by (G) and we already know (E3) holds. To prove (E4), suppose $(e_1,\id{D})\circ (u_1,u_2)$ factors through an object in $\mathcal{Z}$. Without loss of generality, we can assume that this object is $(1,0)$. This gives the rest of the above diagram of solid arrows. Now, $(u_1,u_2)$ indeed factors through $(\id{C},m_2)$ as required in (E3): look at the dashed arrow in the diagram above. The proof of (E5) is similar.
\end{proof}

As a consequence, we obtain:

\begin{theorem}\label{TheA}
$(\C\x \D,\mathbb{C}\times \mathbf{0},\mathbf{1}\times\mathbb{D})$ is a pretorsion theory if and only if every morphism $C\to 1$ in $\mathbb{C}$ with $1\in \1$ is an epimorphism and every morphism $0\to D$ in $\mathbb{D}$ with $0\in \0$ is a monomorphism. That is, if and only if $\D$ is quasi-pointed and $\C$ satisfies the dual condition.
\end{theorem}

\begin{proof} Suppose $(\mathbb{C}\times \mathbf{0},\mathbf{1}\times\mathbb{D})$ is a pretorsion theory. Then any object $(C,D)$ is part of a short exact sequence, which by Lemma~\ref{LemA} can be written out as
$$\xymatrix@=30pt{(C,0)\ar[r]^-{(\id{C},m_2)} & (C,D)\ar[r]^-{(e_1,\id{D})} & (1,D)}$$
where $m_2\colon C\to D$ is a monomorphism and $e_1\colon C\to 1$ is an epimorphism. To prove the converse, we can apply Lemma~\ref{LemA} again and use the above sequence for each $(C,D)$. Thanks to the discussion at the start of this section, this would conclude the proof.
\end{proof}

So in such pretorsion theory, the ``torsion part'' of an object $(C,D)$ is given by $(C,0)$, whereas the ``torsion-free part'' is given by $(1,D)$, for some fixed choice of $0\in \0$ and $1\in \1$. 

\begin{remark}
The property that every morphism going out from an initial object $0$ is a monomorphism, under the presence of a terminal object $1$, is equivalent to the property that the unique morphism $0\to 1$ is a monomorphism. This is a useful generalisation of pointedness in categorical algebra, first emphasized in \cite{Bourn01} (see also \cite{GJ17}).    
\end{remark}

\begin{example}\label{ExaA}
Let $\mathbb{C}=\mathbf{Set}^\mathsf{op}$ and $\mathbb{D}=\mathbf{Set}$. We can think of objects $(X,Y)$ in $\mathbf{Set}^\mathsf{op}\times \mathbf{Set}$ as sets having two types of elements: \emph{positive} elements (elements of $Y$) and \emph{negative} elements (elements of $X$). To emphasize this intuition, we can write $Y-X$ for $(X,Y)$. Theorem~\ref{TheA} is applicable here, since $\varnothing\to S$ is an injective function, for any set $S$. The torsion objects in the corresponding pretorsion theory are \emph{negative sets}, i.e., sets having only negative elements, while torsion-free objects are \emph{positive sets}, i.e., sets having only positive elements. The short exact sequence that every object $Y-X$ fits into identifies the ``subset'' $\varnothing-X$ of $Y-X$ consisting of negative elements, and the subset $Y-\varnothing$ of $X-Y$ consisting of positive elements:
$$\xymatrix{\varnothing-X\ar[r] & Y-X\ar[r] & Y-\varnothing}$$
\end{example}

As shown in \cite{FFG21}, the torsion and the torsion-free part of an object are, in general, unique (up to canonical isomorphisms). In our case, the reverse is also true: two objects having isomorphic torsion and torsion-free pair of objects will necessarily be isomorphic. Indeed, a much stronger property holds, captured by the following definition.

\begin{definition}
A \dfn{rectangular pretorsion theory} is a pretorsion theory $(\C,\T,\F)$ such that the canonical functor 
\begin{fun}
	\Gamma & \: & \C & \too & \T\x \F \\[1ex]
    && \fib{X}{h}{Y} & \to & \fib{(T^X,F^X)}{(h^T,h^F)}{(T^Y,F^Y)},
\end{fun}
assigning to each $X\in \C$ the pair of its torsion and torsion-free parts,
is an equivalence of categories. See \cite{FFG21} for the constructions of the canonical functors $T\:\C\to \T$ and $F\:\C\to\F$ that induce $\Gamma$.
\end{definition}

\begin{proposition}
    Let $\D$ be a quasi-pointed category and $\C$ be a category satisfying the dual condition. The pretorsion theory $(\C\x \D,\mathbb{C}\times \mathbf{0},\mathbf{1}\times\mathbb{D})$ (of \thex\ref{TheA}) is rectangular.
\end{proposition}
\begin{proof}
    The canonical functor $\Gamma$ associated to $(\C\x \D,\mathbb{C}\times \mathbf{0},\mathbf{1}\times\mathbb{D})$ is
    \begin{fun}
	\Gamma & \: & \C\x\D & \too & \mathbb{C}\times \mathbf{0}\x \mathbf{1}\times\mathbb{D} \\[1ex]
    && (X,Y) & \to & ((X,0),(1,Y))
\end{fun}
This is clearly an equivalence of categories, as $\Gamma=T'\x F'$ with  $T'\:X\mapsto (X,0)$ and $F'\:Y\mapsto (1,Y)$, and both $T'$ and $F'$ are equivalences.
\end{proof}

Our construction of the pretorsion theory $(\C\x \D,\mathbb{C}\times \mathbf{0},\mathbf{1}\times\mathbb{D})$ of \thex\ref{TheA} can be generalized significantly, as witnessed by the following results.

\begin{lemma}\label{LemmaZ}
For any category $\mathbb{C}$, the triple $(\C,\C,\J)$ is a pretorsion theory if and only if $\J$ is a full replete epireflective subcategory of $\C$. If $\J$ is equivalent to the terminal category, then $(\C,\C,\J)$ is rectangular.
\end{lemma}
\begin{proof}
    It is known that if we have a pretorsion theory $(\C,\C,\J)$ then $\J$ is a full replete epireflective subcategory of $\C$, see \cite[Corollary 3.4]{FFG21}. Taking the torsion free part of objects of $\C$ gives the epireflection.

    We prove that if $J\:\J\subseteq \C$ is a full replete epireflective subcategory of $\C$ then $(\C,\C,\J)$ is a pretorsion theory. Call $L\:\C\to \J$ the epireflection. Of course both $\C$ and $\J$ are full replete subcategories of $\C$. The ideal of null morphisms is given by $\C\cap \J=\J$, so every morphism landing into an object of $\J$ is null. It remains to show that every object $X\in \C$ sits in the middle of an appropriate short exact sequence. But given $X\in \C$, the epireflection produces an epimorphism $\eta_X\:X\to J(L(X))$, with $L(X)\in \J$. So consider the sequence
    $$X\aeqq{\id{X}} X\aar{\eta_X}J(L(X))$$
    Since $\eta_X$ is null, $\id{X}$ is a kernel of $\eta_X$. We show that $\eta_X$ is the cokernel of $\id{X}$. Given any null morphism $m\:X\to M$, we have that $m$ factors as
    $X\aar{m^1} J(A)\aar{m^2}M$ with $A\in \J$. As $\eta_X$ gives a universal arrow, there exists a unique $v\:L(X)\to A$ such that
    \begin{cd}[4][6]
        X \ar[r,"{\eta_X}"]\ar[rd,bend right,"{m^1}"'] \& J(L(X))\ar[d,"J(v)"] \\
        \& J(A)
    \end{cd}
    Whence $m^2\c J(v)$ is such that $m^2\c J(v)\c \eta_X=m$. Since $\eta_X$ is epi, $m^2\c J(v)$ is the unique morphism with this property. We thus conclude that the sequence above is short exact.

    Now, assume that $\J$ is equivalent to the terminal category, and consider the canonical functor
\begin{fun}
	\Gamma & \: & \C & \too & \C\x \J \\[1ex]
    && \fib{X}{h}{Y} & \to & \fib{(X,L(X))}{(h,L(h))}{(Y,L(Y))}
\end{fun}
Notice that indeed $h^T=h$ and $h^F=L(h)$, by uniqueness of $h^T$ and $h^F$ and naturality of $\eta$. We show that the projection functor $\pi_1\:\C\x \J\to \C$ is a pseudo-inverse of $\Gamma$. Of course $\pi_1\c \Gamma=\Id{}$. Consider now $(X,A)\in \C\x\J$. Since $\J$ is equivalent to the terminal category, there exists an isomorphism $A\iso L(X)$, whence $(X,A)\iso (X,L(X))$. Moreover, such isomorphisms form a natural transformation, again because $\J$ is equivalent to the terminal category. And we conclude that $\Gamma\c \pi_1\iso \Id{}$.
\end{proof}

Dually and symmetrically, we have:

\begin{lemma}\label{LemmaW}
For any category $\mathbb{D}$, the triple $(\D,\I,\D)$ is a pretorsion theory if and only if $\I$ is a full replete monocoreflective subcategory of $\D$. If $\I$ is equivalent to the terminal category, then $(\D,\I,\D)$ is rectangular.
\end{lemma}

\begin{corollary}\label{LemD}
Let $\mathbb{C}$ be a category with terminal object and let $\1$ be the subcategory of terminal objects in $\C$. The triple $(\C,\mathbb{C},\mathbf{1})$ is a pretorsion theory if and only if every morphism in $\C$ to the terminal object is an epimorphism. Such pretorsion theories are rectangular.
\end{corollary}
\begin{proof}
    $\1$ is of course a full replete subcategory of $\C$, by uniqueness of the terminal object up to iso. If every morphism in $\C$ to the terminal object is an epimorphism, then the unique epimorphisms $X\to 1$, for a fixed chosen $1\in \1$, give an epireflection of $\C$ into $\1$. Indeed, terminal objects are unique up to a unique isomorphism. By \lemx\ref{LemmaZ}, $(\C,\C,\1)$ is a rectangular pretorsion theory.
    
    If $(\C,\C,\1)$ is a pretorsion theory, then by \lemx\ref{LemmaZ} $\1$ is epireflective in $\C$. So for every $X\in \C$ there exists an epimorphism $\eta_X\:X\to 1$ with $1\in \1$. Whence every morphism in $\C$ to the terminal object is an epimorphism.
\end{proof}

\begin{corollary}\label{LemC}
Let $\mathbb{D}$ be a category with initial object and let $\0$ be the subcategory of initial objects in $\D$. The triple $(\D,\0,\mathbb{D})$ is a pretorsion theory if and only if every morphism in $\D$ from the initial object is a monomorphism. Such pretorsion theories are rectangular.
\end{corollary}

\begin{remark}
\corx\ref{LemD} (and analogously also \corx\ref{LemC}) can also be derived from Theorem~\ref{TheA}. Let $\D$ be the terminal category $\{\ast\}$ and apply Theorem~\ref{TheA} to $\mathbb{C}$ and $\{\ast\}$. We obtain that every morphism to the terminal object in $\mathbb{C}$ is an epimorphism if and only if $(\mathbb{C}\times \{\ast\},\mathbf{1}\times\{\ast\})$ is a pretorsion theory in $\mathbb{C}\times\{\ast\}$. The statement of \corx\ref{LemD} then follows thanks to the canonical isomorphism $\mathbb{C}\iso\mathbb{C}\times\{\ast\}$.  
\end{remark}

Thanks to Corollaries~\ref{LemD} and \ref{LemC}, we can clarify Theorem~\ref{TheA} using products of pretorsion theories. Indeed, the following result proves (in particular) that we can form products of pretorsion theories.

\begin{theorem}\label{TheB}
Consider a non-empty family $\mathbb{C}=(\mathbb{C}_i)_{i\in I}$ of non-empty categories and a family $(\mathcal{T}_i,\mathcal{F}_i)_{i\in I}$ of pairs of categories. The triple
\begin{equation}\label{EquB}(\Pi\mathbb{C},\Pi\mathcal{T},\Pi\mathcal{F})=\left(\prod_{i\in I}\mathbb{C}_i,\prod_{i\in I}\mathcal{T}_i,\prod_{i\in I}\mathcal{F}_i\right)\end{equation}
is a pretorsion theory if and only if each $(\C_i,\mathcal{T}_i,\mathcal{F}_i)$ is a pretorsion theory. Moreover, $(\Pi\mathbb{C},\Pi\mathcal{T},\Pi\mathcal{F})$ is rectangular if and only if each $(\C_i,\mathcal{T}_i,\mathcal{F}_i)$ is rectangular.
\end{theorem}

\begin{proof} Firstly, we remark that when either of the two conditions hold, each $\mathcal{T}_i,\mathcal{F}_i$ are non-empty. So without loss of generality we may assume that these categories are non-empty from the onset. 
Under this assumption, each $\mathcal{T}_i,\mathcal{F}_i$ are full replete subcategories of $\mathbb{C}_i$ if and only if $\Pi\mathcal{T},\Pi\mathcal{F}$ are full replete subcategories of $\Pi\mathbb{C}$.

We have:
$$\Pi\mathcal{T}\cap \Pi\mathcal{F}=\Pi_{i\in I}(\mathcal{T}_i\cap \mathcal{F}_i).$$
This implies that a morphism is null in $\Pi\mathbb{C}$ if and only if it is null in each $\mathbb{C}_i$. So null-morphisms in $\Pi\mathbb{C}$ are \emph{component-wise}. This clearly guarantees that (T1) holds for $(\Pi\mathcal{T},\Pi\mathcal{F})$ if and only if it holds for each $(\mathcal{T}_i,\mathcal{F}_i)$. To guarantee the same for (T2) it suffices to show that short exact sequences in $\Pi\mathbb{C}$ are also component-wise (an analogue, and in fact a generalisation of Lemma~\ref{LemA}). It is easy to see that if each 
$$\xymatrix{T_i\ar[r]^-{m_i} & C_i\ar[r]^-{e_i} & F_i}$$
is a short exact sequence, then so is 
$$\xymatrix{(T_i)_{i\in I}\ar[r]^-{(m_i)_{i\in I}} & (C_i)_{i\in I}\ar[r]^-{(e_i)_{i\in I}} & (F_i)_{i\in I}.}$$
Conversely, suppose the above is a short exact sequence. To show that the previous sequence is also short exact for some $i\in I$, consider a morphism $u_i\colon U_i\to C$ such that $e_iu_i$ is null. Define $U_j=T_j$ when $j\neq i$ and $u_j=m_j$ when $j\neq i$. Then the composite $(e_i)_{i\in I}(u_i)_{i\in I}$ is null. This results in a factorisation of $(u_i)_{i\in I}$ through $(m_i)_{i\in I}$, and hence a factorisation of $u_i$ through $m_i$. By the fact that being a monomorphism is a component-wise property and duality, we obtain the desired: that the sequence at $i$ shown above is short exact.

Next, we prove that rectangularity is preserved and reflected under products. It is easy to see that the functor $K\colon \Pi\mathbb{C}\to \Pi\mathcal{T}\times \Pi\mathcal{F}$ that assigns to each objects its torsion and torsion-free part can be built component-wise using similar functors at each $K_i$. In fact, we have a commutative diagram
$$\xymatrix{\Pi\mathbb{C}\ar[r]^-{K}\ar[rd]_-{\Pi_{i\in I}(K_i)\;\;\;} & \Pi\mathcal{T}\times \Pi\mathcal{F}\ar[d]^-{L}\\ & \Pi_{i\in I}(\mathcal{T}_i\times \mathcal{F}_i)}$$
where $L$ is an isomorphism. So $K$ will be an equivalence if and only if each $K_i$ is an equivalence. 
\end{proof}

\begin{example}\label{ExaB}
Consider the category $\mathbf{Set}_\mathrm{Rel}$ of sets and relations between sets as morphisms. Then $(\mathbf{Set}_\mathrm{Rel},\mathbf{Set}_\mathrm{Rel})$ is a pretorsion theory. 
Apply Theorem~\ref{TheB} to get a pretorsion theory which is the product of this pretorsion theory with the one from Example~\ref{ExaA}. Objects of the resulting category can be thought of as sets having three types of elements: negative, positive (as in Example~\ref{ExaA}), and neutral (those contributed by $\mathbf{Set}_\mathrm{Rel}$). Torsion part of an object can be obtained by discarding the positive elements, while torsion-free part by discarding the negative elements. 
\end{example}

Between Theorem~\ref{TheA} and the theorem above, there is also the following interesting intermediate result.

\begin{theorem}\label{TheC}
For a category $\mathbb{C}$ with a full replete subcategory $\J$, and a category $\mathbb{D}$ with a full replete subcategory $\I$, the triple 
$(\mathbb{C}\times\mathbb{D},\mathbb{C}\times \I,\J\times\mathbb{D})$ is a pretorsion theory if and only if $\J$ is an epireflective subcategory of $\mathbb{C}$ and $\I$ is a monocoreflective subcategory of $\mathbb{D}$.
\end{theorem}
\begin{proof}
    This theorem is a consequence of \thex\ref{TheB} and Lemmas~\ref{LemmaZ} and \ref{LemmaW}.
\end{proof}

\begin{remark}
Pretorsion theories form a $2$-category $\PTors$ where
\begin{itemize}
\item $1$-cells $(\mathcal{C},\mathcal{T},\mathcal{F})\to (\mathcal{C}',\mathcal{T}',\mathcal{F}')$ are functors $\mathcal{C}\to \mathcal{C}'$ which preserve torsion and torsion-free objects, and send chosen short exact sequences to (not necessarily chosen) short exact sequences.

\item $2$-cells are natural transformations between functors.
\end{itemize}
We will denote by $\CartPTors$ the full sub-2-category of $\PTors$ consisting of rectangular pretorsion theories.
\end{remark}

\begin{theorem}
    The $2$-category $\PTors$ has all products, given by the pretorsion theories (\ref{EquB}) constructed in \thex\ref{TheB}.
\end{theorem}
\begin{proof}
The proof is an easy routine and hence we omit it.
\end{proof}

We now show that the equivalences $\Gamma\:\C\to\T\x\F$ associated to rectangular pretorsion theories $(\C,\T,\F)$ behave well with 1-cells and 2-cells between rectangular pretorsion theories. 

\begin{construction}\label{conslambda}
Let $(\C,\T,\F)$ and $(\D,\T',\F')$ be rectangular pretorsion theories with associated equivalences $\Gamma\:\C \to \T \x \F$ and $\Delta\: \D \to \T'\x \F'$ and let $G:\C \to \D$ be a morphism between them. Consider the short exact sequence 
$$T^{X} \aar{\ell^{X}} X \aar{r^{X}} F^{X}$$ 
associated to $X\in \C$. Since $G$ sends the chosen short exact sequences in $(\mathcal{C},\mathcal{T},\mathcal{F})$ into short exact sequences in $\D$, we obtain a short exact sequence 
%$$G(T^{X}) \ar{G(\ell^{X})} G(X) \ar{G(r^{X})} G(F^{X}),$$
with $G(T^X)\in \T'$ and $G(F^X)\in \F'$.
But short exact sequences from torsion objects to torsion free objects are determined up to (unique) isomorphisms and so we obtain isomorphisms 
$$\lambda_X^1\: T^{G(X)} \aiso{} G(T^X) \qquad \text{and} \qquad \lambda_X^2\: F^{G(X)} \aiso{} G(F^X)$$ such that the following diagram commutes 
\begin{cd}[5][5]
{T^{G(X)}}  \arrow[r, "\ell^{G(X)}"] \arrow[d,aiso, "\lambda_X^1"']\& {G(X)} \arrow[r, "r^{G(X)}"] \arrow[d, "", equal]   \& {F^{G(X)}} \arrow[d,aiso, "\lambda_X^2"'] \\
{G(T^X)} \arrow[r, "G(\ell^X)"']\& {G(X)} \arrow[r, "G(r^X)"']\& {G(F^X)}
\end{cd}

The family 
$\fami{\lambda_X=(\lambda_X^1, \lambda_X^2)\: \Delta(G(X)) \to (G\times G)(\Gamma (X))}{x\in \C}$ gives a natural isomorphism 
\sq[i][6][6][\lambda][3][2.8]{\C}{\D}{\T\x \F}{\T'\x \F'.}{G}{\Gamma}{\Delta}{G\x G}
Indeed, given a morphism $g\: X \to Y$ in $\C$, by the commutativity of the diagrams 
\begin{cd}[5][5]
{T^{G(X)}}  \arrow[r, "\ell^{G(X)}"] \arrow[d,aiso, "\lambda_X^1"']\& {G(X)} \arrow[r, "r^{G(X)}"] \arrow[d, "", equal]   \& {F^{G(X)}} \arrow[d,aiso, "\lambda_X^2"'] \\
{G(T^X)} \arrow[d, "G(g^T)"'] \arrow[r, "G(\ell^X)"']\& {G(X)} \arrow[d,"G(g)"] \arrow[r, "G(r^X)"']\& {G(F^X)} \arrow[d,"G(g^F)"]\\
{T^{G(Y)}}  \arrow[r, "\ell^{G(Y)}"'] \arrow[d,aiso, "\lambda_Y^1"']\& {G(Y)} \arrow[r, "r^{G(Y)}"'] \arrow[d, "", equal]   \& {F^{G(Y)}} \arrow[d,aiso, "\lambda_Y^2"'] \\
{G(T^Y)} \arrow[r, "G(\ell^Y)"']\& {G(Y)} \arrow[r, "G(r^Y)"']\& {G(F^Y)}
\end{cd}
we conclude, by uniqueness of $G(g)^F$ and $G(g)^T$, that 
$$G(g)^T=(\lambda_Y^1)^{-1} \c G(g^{T}) \c \lambda_X^1 \qquad \text{ and } \qquad G(g)^F=(\lambda_Y^2)^{-1} \c G(g^{F}) \c \lambda_X^2.$$  
Hence we obtain that the diagram 
\sq{\Delta(G(X))}{(G\x G)(\Gamma(X))}{\Delta(G(Y))}{(G\x G)(\Gamma(Y))}{\lambda_X}{\Delta(G(g))}{(G\x G)(\Gamma(g))}{\lambda_Y}
is commutative since its components are. And so $\lambda$ is a natural transformation.
\end{construction}

\begin{proposition}\label{propalphacomp}
    Let $(\C,\T,\F)$ and $(\D,\T',\F')$ be rectangular pretorsion theories and let $\Gamma\: \C \to \T \x \F$ and $\Delta\: \D \to \T'\x \F'$ be the corresponding equivalences of categories. Let then $G,H\:(\C,\T,\F)\to (\D,\T',\F')$ be morphisms of rectangular pretorsion theories and let $\alpha\:G \Rightarrow H$ be a 2-cell between them. The following equality holds
    \begin{eqD*}
        \begin{cd}*
		{\C} \arrow[d,"{\Gamma}"',] \arrow[r,"{G}",""'{name=D}, bend left=65]  \arrow[r,"{H}"',""'{name=E}] \arrow[Rightarrow,from=D, to=E,"{\alpha}"{pos=0.35}, shorten >=0.3ex]\& {\D} \arrow[ld, Rightarrow,twoiso,"{\rho}", shorten <=2.5ex, shorten >=3ex,pos=0.4] \arrow[d,"{\Delta}"] \\
		{\T \x \F} \arrow[r,"{H \x H}"']\& {\T'\x \F'}
	\end{cd}
	\h[3]=\h[5] 
	\begin{cd}*
		{\C} \arrow[d,"{\Gamma}"',] \arrow[r,"G"] \& {\D} \arrow[ld, Rightarrow,twoiso,"{\lambda}", shorten <=2.5ex, shorten >=3ex,pos=0.4]  \arrow[d,"{\Delta}"]\\
		{\T \x \F} \arrow[r,"{G\x G}", ""'{name=A}] \arrow[r,"{H \x H}"', bend right=65,""{name=B}] \arrow[Rightarrow,from=A, to=B,"{\alpha \x \alpha}"{pos=0.75}, shorten >=-0.6ex, shorten <=-0.3ex]\& {\T' \x \F',}
	\end{cd}
\end{eqD*}
    where $\lambda$ and $\rho$ are the natural transformations associated to $G$ and $H$ respectively as in \conx\ref{conslambda}.
\end{proposition}

\begin{proof}
    The desired equality can be proved on components and checked postcomposing with the production projections. So, given $X\in \C$, it suffices to prove that $\rho_X^{1} \c \alpha_X^T= \alpha_{T^X} \c \lambda_X^1$ and $\rho_X^{2} \c \alpha_X^F= \alpha_{F^X} \c \lambda_X^2$. Notice now that the following diagrams are commutative
    
\[\begin{tikzcd}
	{T^{G(X)}} & {G(X)} & {F^{G(X)}} && {T^{G(X)}} & {G(X)} & {F^{G(X)}} \\
	{T^{H(X)}} & {H(X)} & {F^{H(X)}} && {G(T^X)} & {G(X)} & {G(F^X)} \\
	{H(T^X)} & {H(X)} & {H(F^X)} && {H(T^X)} & {H(X)} & {H(F^X).}
	\arrow["{\ell^{G(X)}}", from=1-1, to=1-2]
	\arrow["{\alpha_X^T}"', from=1-1, to=2-1]
	\arrow["{r^{G(X)}}", from=1-2, to=1-3]
	\arrow["{\alpha_X}", from=1-2, to=2-2]
	\arrow["{\alpha_X^F}", from=1-3, to=2-3]
	\arrow["{\ell^{G(X)}}", from=1-5, to=1-6]
	\arrow["{\lambda^1_X}"', from=1-5, to=2-5]
	\arrow["{r^{G(X)}}", from=1-6, to=1-7]
	\arrow[from=1-6, to=2-6, equal]
	\arrow["{\lambda^2_X}", from=1-7, to=2-7]
	\arrow["{\ell^{H(X)}}"', from=2-1, to=2-2]
	\arrow["{\rho^1_X}"', from=2-1, to=3-1]
	\arrow["{r^{H(X)}}"', from=2-2, to=2-3]
	\arrow[from=2-2, to=3-2, equal]
	\arrow["{\rho^2_X}", from=2-3, to=3-3]
	\arrow["{G(\ell^X)}"', from=2-5, to=2-6]
	\arrow["{\alpha_{T^X}}"', from=2-5, to=3-5]
	\arrow["{G(r^X)}"', from=2-6, to=2-7]
	\arrow["{\alpha_X}", from=2-6, to=3-6]
	\arrow["{\alpha_{F^X}}", from=2-7, to=3-7]
	\arrow["{H(\ell^X)}"', from=3-1, to=3-2]
	\arrow["{H(r^X)}"', from=3-2, to=3-3]
	\arrow["{H(\ell^X)}"', from=3-5, to=3-6]
	\arrow["{H(r^X)}"', from=3-6, to=3-7]
\end{tikzcd}\]
Since $H(\ell^X)$ is a monomorphism, the commutativity of diagram on the left implies $\rho_X^{1} \c \alpha_X^T= \alpha_{T^X} \c \lambda_X^1$. Moreover, since $r^{G(X)}$ is an epimorphism, the commutativity of the diagram on the right implies $\rho_X^{2} \c \alpha_X^F= \alpha_{F^X} \c \lambda_X^2$. 
\end{proof}

While \thex\ref{TheB} implies that products of rectangular pretorsion theories are rectangular, it does not characterize all rectangular pretorsion theories. The following results will allow us to find such characterization. %We will show that they are, up to equivalence, of the form given by Theorem~\ref{TheA}. In other words, they are products of torsion theories described by Corollaries~\ref{LemD} and \ref{LemC}.

\begin{lemma}\label{lemmakercoker}
    Let $(\C,\T,\F)$ be a pretorsion theory. Given a kernel $k\:k(X)\to X$ of $\id{X}$, we have that $k(X)\in \T\cap \F$. Analogously, given a cokernel $c\:X\to c(X)$ of $\id{X}$, we have that $c(X)\in \T\cap \F$.
\end{lemma}
\begin{proof}
    Let $k\:k(X)\to X$ be a kernel of $\id{X}$. Then $k$ needs to be a null morphism, so $k$ factors as
$$k(X)\aar{j}Z\aar{k'}X$$
with $Z\in \T\cap \F$. Since $\id{X}\c k'$ is null, there exists a unique morphism $v\:Z\to k(X)$ such that $k\c v=k'$. Then $v\c j=\id{}$, because this holds after composing with the mono $k$. So $k(X)$ is a retract of $Z\in \T\cap \F$, and by \cite[Corollary 2.8]{FFG21} we have $k(X)\in \T\cap \F$.

Analogously, we can prove that $c(X)\in \T\cap \F$.
\end{proof}

\begin{lemma}\label{LemE}
Let $(\C,\mathcal{T},\mathcal{F})$ be a rectangular pretorsion theory where $\C$ is a non-empty category. Then all objects in $\mathcal{T}\cap\mathcal{F}$ are isomorphic to an object $Z$ which is terminal in $\mathcal{T}$ and initial in $\mathcal{F}$. In particular, the ideal of null morphisms is generated by one object $Z$. Moreover $\F$ is quasi-pointed and $\T$ satisfies the dual condition.
\end{lemma}
\begin{proof}
Consider an object $Z\in\mathcal{T}\cap\mathcal{F}$ (such $Z$ exists by Remark 2.9(a) in \cite{FFG21}). The canonical functor $\mathbb{C}\to\mathcal{T}\times\mathcal{F}$ maps $Z$ to an object in $\mathcal{T}\times\mathcal{F}$ that is isomorphic to $(Z,Z)$. Since this functor is an equivalence, we obtain a bijection
$$\mathsf{hom}(Z,Z)\iso \mathsf{hom}((Z,Z),(Z,Z))$$
given by the mapping $z\mapsto (z,z)$. This forces $\mathsf{hom}(Z,Z)$ to be a singleton. Indeed, given $(z'\:Z\to Z, z''\:Z\to Z)$, there exists a unique $z\:Z\to Z$ such that $(z,z)=(z',z'')$. Thus $\mathsf{hom}(Z,Z)$ only contains the identity of $Z$. Now, consider another object $Z'$ in the intersection $\mathcal{T}\cap\mathcal{F}$. By the canonical equivalence $\mathbb{C}\to\mathcal{T}\times\mathcal{F}$, there must be an object $C$ in $\mathbb{C}$ with the torsion and torsion-free parts given by 
$T^C=Z$ and $F^C=Z'$. This forces isomorphisms $Z\iso C\iso Z'$, because identities give kernels and cokernels of null morphisms. Whence all objects in $\T\cap \F$ are isomorphic to $Z$.

We prove that $Z$ is terminal in $\T$ and initial in $\F$. Given an object $T\in\mathcal{T}$, consider its reflection $r^T\:T\to F^T$ of $T$ to its torsion-free part $F^T$. Since we always choose $\ell^T=\id{T}$ (see the Notations subsection), $r^T$ is a cokernel of $\id{T}$. By \lemx\ref{lemmakercoker}, $F^T\in \T\cap \F$ and thus $F^T\iso Z$. In other words, for every $T\in \T$ we have a morphism $T\to Z$ that is a cokernel of $\id{T}\colon T\to T$ in $\C$. As a consequence, it is also a cokernel of $\id{T}$ in $\T$ and in particular an epi in $\T$. Moreover, since $\mathsf{hom}(Z,Z)$ is a singleton, there can only be one morphism $T\to Z$, proving that $Z$ is a terminal object in $\mathcal{T}$. Dually, $Z$ is an initial object in $\mathcal{F}$ such that all morphisms going out from $Z$ in $\F$ are mono in $\F$, which completes the proof.
\end{proof}

\begin{definition}\label{defequiv}
    By an \emph{equivalence} 
$(\mathcal{C},\mathcal{T},\mathcal{F})\approx (\mathcal{D},\mathcal{T}',\mathcal{F}')$
of pretorsion theories we mean an equivalence $\mathcal{C}\approx\mathcal{D}$ of categories which preserves torsion and torsion-free objects.
\end{definition}

\begin{lemma}\label{lemmaequiv}
    Let $\C$ and $\D$ be categories equipped with an ideal of null morphisms. Every equivalence of categories $\Lambda\:\C\approx \D:\Lambda'$ that preserves null morphisms preserves also kernels and cokernels, and thus also short exact sequences.
\end{lemma}
\begin{proof}
    We can assume that the equivalence $\Lambda\:\C\approx \D:\Lambda'$ is an adjoint equivalence. Consider a composite of morphisms
    $$T\aar{l}X\aar{r}F.$$
    Assume that $\ell$ is the kernel of $r$. We prove that $\Lambda(\ell)$ is the kernel of $\Lambda(r)$. So consider a morphism $m\:M\to \Lambda(X)$ such that $\Lambda(r)\c m$ is a null morphism. By assumption, also $\Lambda'(\Lambda(r))\c \Lambda'(m)$ is a null morphism, whence also
    $$\Lambda'(M)\aar{\Lambda'(m)}\Lambda'(\Lambda(X))\iso X\aar{r} F$$
    is a null morphism. Since $\ell$ is the kernel of $r$, there exists a unique morphism $v\: \Lambda'(M)\to X$ such that
    \begin{cd}[5][5]
        \Lambda'(M) \ar[r,"\Lambda'(m)"]\ar[d,"v"']\& \Lambda'(\Lambda(X)) \ar[d,iso]\\
        T \ar[r,"\ell"']\& X
    \end{cd}
Whence the morphism $w$ given by the composite $M\iso \Lambda(\Lambda'(M))\aar{\Lambda(v)} \Lambda(T)$ is such that $\Lambda(l)\c w=m$, thanks to the triangular identity of the adjoint equivalence. Since $\Lambda(l)$ is mono, because every equivalence preserves limits, $w$ is the unique morphism with such property. We thus conclude that $\Lambda(\ell)$ is the kernel of $\Lambda(r)$.

Analogously, we have that $\Lambda$ preserves cokernels.
\end{proof}

\begin{proposition}\label{propequivpretors}
    Equivalences of pretorsion theories $(\C,\T,\F)\approx (\D,\T',\F')$, as defined in \defx\ref{defequiv}, coincide with internal equivalences in the $2$-category $\PTors$ of pretorsion theories.
\end{proposition}
\begin{proof}
    Every equivalence of categories $\Lambda\:\C\approx \D\:\Lambda'$ that preserves torsion and torsion-free objects also preserves null morphisms, which are generated by $\T\cap \F$. By \lemx\ref{lemmaequiv}, $\Lambda$ and $\Lambda'$ then also preserve short exact sequences. It is straightforward to conclude.
\end{proof}

\begin{corollary}\label{corolltransfer}
    Let $\Lambda\:\C\approx \D\:\Lambda'$ be an equivalence of categories and let $(\D,\T',\F')$ be a pretorsion theory. Then there exists a pretorsion theory $(\C,\T,\F)$ such that $\Lambda$ lifts to an equivalence of pretorsion theories.
\end{corollary}
\begin{proof}
    Define $\T$ to be the full subcategory of $\C$ on the objects of the essential image of 
    $$\T'\cont \D\aar{\Lambda'}\C,$$
    and analogously define $\F$ using $\F'\cont \D\aar{\Lambda'}\C$. Equivalently, $\T=\Lambda^{-1}(\T')$ and $\F=\Lambda^{-1}(\F')$. Then of course $\T$ and $\F$ are replete, and both $\Lambda$ and $\Lambda'$ preserve torsion and torsion-free objects. As a consequence, $\Lambda$ and $\Lambda'$ also preserve null morphisms, and every morphism in $\C$ from an object of $\T$ to an object of $\F$ is null. Finally, given $X\in \C$, we want to construct an appropriate short exact sequence with $X$ in the middle. We consider $\Lambda(X)$ and its associated short exact sequence
    $$T^{\Lambda(X)}\aar{\ell}\Lambda(X)\aar{r}F^{\Lambda(X)}$$
    in $\D$. Since $\Lambda'$ is an equivalence of categories that preserves null morphisms, by \lemx\ref{lemmaequiv} it preserves short exact sequences. We thus obtain a short exact sequence
$$\Lambda'(T^{\Lambda(X)})\aar{\Lambda'(\ell)}\Lambda'(\Lambda(X))\iso X\iso \Lambda'(\Lambda(X))\aar{\Lambda'(r)}\Lambda'(F^{\Lambda(X)}).$$
And we conclude that $(\C,\T,\F)$ is a pretorsion theory such that $\Lambda$ lifts to an equivalence of pretorsion theories (also thanks to \prox\ref{propequivpretors}).
\end{proof}

\begin{corollary}\label{corequivtf}
    Let $(\C,\T,\F)$ be a rectangular pretorsion theory, and consider its associated equivalence $\Gamma\:\C\to \T\x\F$. Then $$(\T\x\F,\hspace{0.1ex}\T\x(\T\cap \F),(\T\cap\F)\x \F)$$
    is a pretorsion theory equivalent to $(\C,\T,\F)$ via $\Gamma$.
\end{corollary}
\begin{proof}
    By \corx\ref{corolltransfer}, there exists a pretorsion theory $(\T\x\F,\T'',\F'')$ such that $\Gamma$ lifts to an equivalence of pretorsion theories. By the proof of \corx\ref{corolltransfer}, $\T''$ is given by the full subcategory of $\T\x\F$ on the objects of the essential image of $\T\cont \C\aar{\Gamma} \T\x \F$. For every $T\in \T$, $\Gamma(T)=(T^T,F^T)\iso (T,Z)$ for fixed chosen $Z\in \T\cap \F$ such that all objects in $\T\cap \F$ are isomorphic to $Z$, by \lemx\ref{LemE} and its proof. So $\T''=\T\x \T\cap \F=\T\x\mathbf{Z}$, where $\mathbf{Z}$ is the full subcategory of $\C$ given by the objects isomorphic to $Z$. Analogously, $\F''=\T\cap \F\x \F=\mathbf{Z}\x \F$.
\end{proof}

\begin{theorem}\label{theorcharactcart}
For a pretorsion theory $(\mathcal{C},\mathcal{T},\mathcal{F})$ where $\mathcal{C}$ is a non-empty category, the following conditions are equivalent:
\begin{itemize}
\item[(a)] $(\mathcal{C},\mathcal{T},\mathcal{F})$ is a rectangular pretorsion theory.

\item[(b)] $(\mathcal{C},\mathcal{T},\mathcal{F})$ is equivalent to a pretorsion theory of the form $(\C\x \D,\mathbb{C}\times \mathbf{0},\mathbf{1}\times\mathbb{D})$ where $\D$ is a quasi-pointed category and $\C$ is a category satisfying the dual condition, as described in Theorem~\ref{TheA}.

\item[(c)] $(\mathcal{C},\mathcal{T},\mathcal{F})$ is equivalent to a product of pretorsion theories of the form $(\C,\C,\1)$ and $(\D,\0,\D)$ described in Corollaries~\ref{LemD} and \ref{LemC}, where $\D$ is a quasi-pointed category and $\C$ is a category satisfying the dual condition.
\end{itemize}
\end{theorem}

\begin{proof} First, let us note that (b)$\Rightarrow$(c) is obvious and (c)$\Rightarrow$(a) holds by Theorem~\ref{TheA}. It then suffices to prove
(a)$\Rightarrow$(b). Suppose $(\mathcal{C},\mathcal{T},\mathcal{F})$ is a rectangular pretorsion theory. By \corx\ref{corequivtf}, the canonical equivalence $\Gamma\:\mathcal{C}\to \mathcal{T}\times\mathcal{F}$ lifts to an equivalence of pretorsion theories between $(\C,\T,\F)$ and $$(\T\x\F,\hspace{0.1ex}\T\x(\T\cap \F),(\T\cap\F)\x \F).$$
By \lemx\ref{LemE}, all objects in $\T\cap \F$ are isomorphic to an object $Z$ which is terminal in $\T$ and initial in $\F$. Moreover, again by \lemx\ref{LemE}, $\F$ is quasi-pointed and $\T$ satisfies the dual condition. So $(\T\x\F,\hspace{0.1ex}\T\x(\T\cap \F),(\T\cap\F)\x \F)$ is in fact of the form described in Theorem~\ref{TheA}.
%It follows from the same lemma that we have an equivalence of categories
%$$(\mathcal{T}\times(\mathcal{T}\cap \mathcal{F}))\times((\mathcal{T}\cap\mathcal{F})\times\mathcal{F})\approx\mathcal{T}\times\mathcal{F},\quad ((T,Z),(Z',F))\mapsto (T,F).$$
%Applying Lemma~\ref{LemE} one more time we get that the pretorsion theory inherited by $(\mathcal{T}\times(\mathcal{T}\cap\mathcal{F}))\times((\mathcal{T}\cap\mathcal{F})\times\mathcal{F})$ under this equivalence is in fact of the form described in Theorem~\ref{TheA}.
\end{proof}

\section{The pointed case}

In this section, we focus more on the pointed version of rectangular pretorsion theories, that we call ``rectangular torsion theories''. Interestingly, a rectangular torsion theory is automatically a torsion theory in the sense of \cite{GJ20} (the same is not true for non-rectangular pointed pretorsion theories). We aim to prove a monadicity result for rectangular torsion theories. We will present this as a corollary of a more general monadicity result, for a symmetrical version of rectangular pretorsion theories.

\begin{definition}
A \predfn{rectangular torsion theory} is a rectangular pretorsion theory $(\C,\T,\F)$ where $\C$ is a pointed category, i.e.\ has zero object $0$, and $\mathcal{Z}=\T\cap\F=\0$, where $\0$ is the subcategory of zero objects of $\C$.
\end{definition}

\begin{proposition}\label{proppointedcase}
    For a pretorsion theory $(\C,\T,\F)$ where $\C$ is a non-empty category, the following conditions are equivalent:
    \begin{itemize}
        \item[(a)] $(\C,\T,\F)$ is a rectangular torsion theory.
        \item[(b)] $(\C,\T,\F)$ is equivalent to a product of pretorsion theories $(\C,\C,\0)$ and $(\D,\0,\D)$ described in Corollaries~\ref{LemD} and \ref{LemC}, where both $\C$ and $\D$ are pointed categories.
    \end{itemize}
\end{proposition}
\begin{proof}
    We prove $(a)\aR{}(b)$. By \lemx\ref{LemE}, all objects in $\mathcal{T}\cap\mathcal{F}$ are isomorphic to an object $Z$ which is terminal in $\mathcal{T}$ and initial in $\mathcal{F}$. Since $(\C,\T,\F)$ is a rectangular torsion theory, $Z$ needs to be a zero object $0$ of $\C$. But then $0\in \T\cont \C$, whence it needs to be an initial object of $\T$. And we had already proved that $0$ is terminal in $\T$. So $0$ is a zero object in $\T$, and analogously also a zero object in $\F$. We conclude by \thex\ref{theorcharactcart}.

    We now prove $(b)\aR{}(a)$. Of course, by \thex\ref{theorcharactcart}, the product $(\C\x \D,\mathbb{C}\times \mathbf{0},\mathbf{0}\times\mathbb{D})$ of $(\C,\C,\0)$ and $(\D,\0,\D)$ is a rectangular pretorsion theory. We then have that $(0,0)$ is a zero object of $\C\x\D$, because it is a zero object componentwise. So $\C\x\D$ is pointed. Finally, notice that $(\C\x \0) \cap (\0\x \D)=\0\x \0$, which coincides with the subcategory of zero objects of $\C\x\D$.
\end{proof}

We can similarly establish pointed version of Theorem~\ref{theorcharactcart}:

\begin{theorem}\label{theorcharactcartpointed}
For a torsion theory $(\mathcal{C},\mathcal{T},\mathcal{F})$, the following conditions are equivalent:
\begin{itemize}
\item[(a)] $(\mathcal{C},\mathcal{T},\mathcal{F})$ is a rectangular torsion theory.

\item[(b)] $(\mathcal{C},\mathcal{T},\mathcal{F})$ is equivalent to a torsion theory of the form $(\C\x \D,\mathbb{C}\times \mathbf{0},\mathbf{0}\times\mathbb{D})$ where $\D$ and $\C$ are pointed categories.

\item[(c)] $(\mathcal{C},\mathcal{T},\mathcal{F})$ is equivalent to a product of torsion theories of the form $(\C,\C,\0)$ and $(\D,\0,\D)$.
\end{itemize}
\end{theorem}

\begin{proposition}\label{propGJpointed}
Every torsion theory is a torsion theory in the sense of \cite{GJ20}. A torsion theory in the sense of \cite{GJ20} that is a rectangular pretorsion theory is automatically a rectangular torsion theory.
\end{proposition}
\begin{proof}
    Of course, every torsion theory in our sense is a torsion theory in the sense of \cite{GJ20}. Let $(\C,\T,\F)$ be a pretorsion theory in the sense of \cite{GJ20} which is rectangular. We thus have kernels $k(X)\aar{k}X$ and cokernels $X\aar{c}c(X)$ of all identities $\id{X}$. We prove that $(\C,\T,\F)$ is a torsion theory. By \lemx\ref{lemmakercoker}, $k(X)$ and $c(X)$ need to be objects in $\T\cap \F$. But by \lemx\ref{LemE}, all objects in $\mathcal{T}\cap\mathcal{F}$ are isomorphic to a fixed object $Z$ which is terminal in $\mathcal{T}$ and initial in $\mathcal{F}$. So both $k(X)$ and $c(X)$ are isomorphic to $Z$. We prove that $Z$ is a zero object in $\C$. By the argument above, for every $X\in \C$, there exists a morphism $k\:Z\to X$ which is the kernel of $\id{X}$. Let $r$ be another morphism $Z\to X$. Since $\id{}\c r$ is null, there exists a unique morphism $w\:Z\to Z$ such that $k\c w=r$. But by the proof of \lemx\ref{LemE}, the only morphism from $Z$ to $Z$ is the identity. So that $k=r$. We conclude that $Z$ is initial in $\C$. Analogously, $Z$ is terminal in $\C$, as it gives the cokernels of all identities.
\end{proof}

\begin{remark}
    Let $(\C,\T,\F)$ be a rectangular pretorsion theory. We showed in the proof of \lemx\ref{LemE} that $\T$ has cokernels of all identities, and $\F$ has kernels of all identities. These are given by some fixed $Z$ such that all objects in $\mathcal{T}\cap\mathcal{F}$ are isomorphic to $Z$ and such that $Z$ is terminal in $\T$ and initial in $\F$. However, $\C$ only has kernels and cokernels of all identities when it is pointed, thanks to \prox\ref{propGJpointed}.
\end{remark}

\begin{definition}
    A rectangular pretorsion theory $(\mathcal{C},\mathcal{T},\mathcal{F})$ is called \dfn{symmetrical} if it is equivalent to a product of pretorsion theories $(\C,\C,\1)$ and $(\D,\0,\D)$ described in Corollaries~\ref{LemD} and \ref{LemC}, where both $\C$ and $\D$ are quasi-pointed and both satisfy the dual condition.
\end{definition}

    Symmetrical rectangular pretorsion theories form a 2-category $\SCartPTors$ where 
\begin{itemize}
\item $1$-cells $(\mathcal{C},\mathcal{T},\mathcal{F})\to (\mathcal{C}',\mathcal{T}',\mathcal{F}')$ are 1-cells between rectangular pretorsion theories that also preserve $0$ and $1$. 

\item $2$-cells are 2-cells between 1-cells of pretorsion theories, i.e.\ natural transformations between functors.
\end{itemize}

\begin{definition}
    We call \dfn{bi-quasi-pointed} those categories that are quasi-pointed and also satisfy the dual condition.
\end{definition}

    Bi-quasi-pointed categories form a 2-category $\Biqpointcat$ where 
    \begin{itemize}
\item $1$-cells are functors that preserve $0$ and $1$.
\item $2$-cells are natural transformations between those functors.
\end{itemize}

\begin{remark}\label{rembiquasipointed}
    A category is bi-quasi-pointed if and only if it has initial object 0 and terminal object 1, and the unique morphism $0\to 1$ is both mono and epi. Of course, every pointed category is bi-quasi-pointed.
\end{remark}

\begin{proposition}\label{propcharactsymm}
    For a rectangular pretorsion theory $(\C,\T,\F)$ where $\C$ is a non-empty category, the following conditions are equivalent:
    \begin{itemize}
        \item[(a)] $(\C,\T,\F)$ is symmetrical.
        \item[(b)] $\C$ is bi-quasi-pointed.
    \end{itemize}
\end{proposition}
\begin{proof}
    We prove $(a)\aR{}(b)$. Assume that $(\C,\T,\F)$ is equivalent to a product of pretorsion theories $(\C,\C,\1)$ and $(\D,\0,\D)$ where both $\C$ and $\D$ are bi-quasi-pointed. The pair $(0_{\C},0_{\D})$ with $0_{\C}$ initial in $\C$ and $0_{\D}$ initial in $\D$ is initial in $\C\x\D$. Analogously, $(1_{\C},1_{\D})$ is terminal in $\C\x\D$. The unique morphism $(0_{\C},0_{\D})\to (1_{\C},1_{\D})$ is both mono and epi, because it is so componentwise.

    We now prove $(b)\aR{}(a)$. By the proof of \thex\ref{theorcharactcart}, $(\C,\T,\F)$ is equivalent to the product of pretorsion theories $(\T,\T,\T\cap \F)$ and $(\F,\T\cap\F,\F)$, and we know that $\F$ is quasi-pointed and $\T$ satisfies the dual condition. By assumption, $\C$ is bi-quasi-pointed. We show that $0\in \T$ and $1\in \F$. By \lemx\ref{LemE}, all objects in $\T\cap \F$ are isomorphic to an object $Z$ which is terminal in $\T$ and initial in $\F$. Consider the short exact sequence $T^0\aar{\ell} 0\aar{r}F^0$ associated to $0$. The morphism $r$ is the unique morphism from $0$ to $F^0$, so it needs to coincide with the composite $0\aar{r^1} Z\aar{r^2} F^0$, where $r^1$ is given by the fact that $0$ is initial in $\C$ and $r^2$ by the fact that $Z$ is initial in $\F$. Then $r$ is null, whence $\ell$ needs to be an isomorphism (as identities give kernels of null maps). So $0\in \T$. Analogously, $1\in \F$. We thus have that $\T$ has initial object given by $0$ and terminal object given by $Z$. By the proof of \lemx\ref{LemE}, all morphisms in $\T$ going into $Z$ are epi. Moreover, all morphisms in $\T$ going out from $0$ are mono, because this is true in $\C$ by assumption. So $\T$ is bi-quasi-pointed. Analogously, $\F$ is bi-quasi-pointed. And we conclude.
\end{proof}

\begin{proposition}\label{propcatcarttors}
    Every rectangular torsion theory is a symmetrical rectangular pretorsion theory. Moreover, the full sub-2-category of $\CartPTors$ on rectangular torsion theories coincides with the full sub-2-category of $\SCartPTors$ of rectangular torsion theories.
\end{proposition}
\begin{proof}
    Every pointed category is bi-quasi-pointed, by \remx\ref{rembiquasipointed}, so every rectangular torsion theory is symmetrical thanks to \prox\ref{propcharactsymm}. Moreover, every 1-cell in $\CartPTors$ between rectangular torsion theories preserves the zero object. Indeed, it needs to preserve torsion objects and torsion-free objects, whence also the objects in the intersection $\T\cap\F=\0$. So every every 1-cell in $\CartPTors$ between rectangular torsion theories is a 1-cell in $\SCartPTors$.
\end{proof}

We write $\CartTors$ for the full sub-2-category of $\CartPTors$ of rectangular torsion theories.

We prove the following monadicity result for symmetrical rectangular pretorsion theories. This will then yield a monadicity result for rectangular torsion theories. We refer the reader to \cite{L02} for the definition of $2$-monad and of the $2$-category of pseudo-algebras for a $2$-monad.

\begin{theorem}\label{teormonad}
    The 2-category $\SCartPTors$ of symmetrical rectangular pretorsion theories is 2-equivalent, over $\Biqpointcat$, to the 2-category $\Alg+{M}$ of pseudo-algebras (and pseudo-morphisms between them) for the (strict) 2-monad
    \begin{fun}
	M & \: & \Biqpointcat & \too & \Biqpointcat \\[1ex]
    && \tcv*{\C}{\D}{G}{H}{\alpha} & \mto & \tcv*{\C\x\C}{\D\x\D}{G\x G}{H\x H}{\alpha\x \alpha}
\end{fun}
\end{theorem}
\begin{proof}
We first prove that $M$ is a (strict) 2-monad. If $\C$ is a bi-quasi-pointed category, also $\C\x\C$ is such, with $0=(0,0)$ and $1=(1,1)$. The unique map from $(0,0)$ to $(1,1)$ is componentwise mono and epi, and thus mono and epi. If $G$ preserves $0$ and $1$ then also $G\x G$ preserves $0$ and $1$. It is straightforward to see that $M$ is a $2$-functor. We define the unit $\eta\:\Id{}\aR{}M$ to have component on $\C$ given by the diagonal functor $\eta_\C\:\C\to \C\x\C$. Of course $\eta_\C$ preserves $0$ and $1$, and $\eta$ is 2-natural. We then define the multiplication $\mu\:M^2\aR{}M$ to have component on $\C$ given by the projection functor
$$\mu_\C=\pi_{1,4}=\pi_1\x\pi_2\:(\C\x\C)\x(\C\x\C)\to \C\x\C$$
on the first and fourth components. Of course $\mu_C$ preserves $0$ and $1$. It is straightfoward to show that $\mu$ is a $2$-natural transformation. The unit axioms of $2$-monad 
\begin{cd}[4][6]
    M \ar[rd,bend right=20,equal]\ar[r,Rightarrow,"{M\eta}"]\& M^2 \ar[d,Rightarrow,"{\mu}"]\& M \ar[l,Rightarrow,"{\eta M}"'] \ar[ld,bend left=20,equal]\\
    \& M
\end{cd}
are satisfied because for every $X,Y\in \C\x\C$ we have
$$\pi_{1,4}((X,X),(Y,Y))=(X,Y)=\pi_{1,4}((X,Y),(X,Y))$$ (and analogously for morphisms). The other axiom
\begin{cd}[5.5][5.5]
    M^3 \ar[r,Rightarrow,"{M \mu}"] \ar[d,Rightarrow,"{\mu M}"']\& M^2 \ar[d,Rightarrow,"{\mu}"] \\
    M^2 \ar[r,Rightarrow,"{\mu}"']\& M
\end{cd}
is satisfied because for every $(X_1,\ldots,X_8)$ we have $$\pi_{1,4}((X_1,X_4),(X_5,X_8))=(X_1,X_8)=\pi_{1,4}((X_1,X_2),(X_7,X_8)).$$
So $M$ is a $2$-monad.

%objects
We now prove that the 2-category $\SCartPTors$ is 2-equivalent over $\Biqpointcat$ to the 2-category $\Alg+{M}$ of pseudo-algebras of $M$. Notice that we have a forgetful $2$-functor $\SCartPTors\to \Biqpointcat$, thanks to \prox\ref{propcharactsymm}.

We start showing that every symmetrical rectangular pretorsion theory $(\C,\T,\F)$ gives a pseudo-algebra for $M$. Let $\Gamma\: \C \to \T \x \F\:\Gamma'$ be the adjoint equivalence of categories associated to $(\C,\T,\F)$ (it can always be made into an adjoint equivalence). We construct a pseudo-algebra
$$(\C,Q\: \C \x \C \to \C, Q_\mu\: Q \circ (Q \x Q) \Rightarrow Q \circ \pi_{1,4}, Q_\eta\: \id{\C} \Rightarrow Q \circ \eta_{\C}).$$ We define the functor $Q\: \C \to \C \x \C$ as the composite 
$$\C \x \C \aar{\Gamma \x \Gamma} (\T \x \F) \x (\T \x \F) \aar{\pi_{1,4}} \T \x \F \aar{\Gamma'} \C.$$
So $Q$ sends $(X,Y)\in \C \x \C$ to $\Gamma'(T^X, F^Y)$. Notice that $Q$ preserves $0$ and $1$. Indeed, $Q(0,0)=\Gamma'(0,Z)\iso 0$, because $\Gamma'(0,Z)$ sits in a short exact sequence $0\aar{l} \Gamma'(0,Z) \aar{r} Z$ and $r$ is null so $l$ needs to be an isomorphism. Analogously $Q(1,1)= \Gamma'(Z,1)\iso 1$. 
We then define $Q_\mu\: Q \circ (Q \x Q) \Rightarrow Q \circ \pi_{1,4}$ as the pasting
\begin{eqD*}
\scalebox{0.8}{
\begin{cd}*
    {(\C \x \C) \x (\C\x \C)} \arrow[r,"(\Gamma\x \Gamma)\x (\Gamma \x \Gamma)"{inner sep= 1.5ex}] \arrow[d,"\pi_{1,4}"'] \& {((\T \x F)\x (\T \x F)) \x ((\T \x F) \x (\T \x F))} \arrow[r,"\pi_{1,4} \x \pi_{1,4}"{inner sep= 1.5 ex}] \& {(\T \x F)\x (\T \x F)} \arrow[d,"\Gamma' \x \Gamma'"] \arrow[dd,"", equal, bend right= 60, shift right =5ex]\\
    {\C \x \C} \arrow[r,"\Gamma \x \Gamma"] \& {(\T \x \F) \x (\T \x \F)} \arrow[dd,"\pi_{1,4}"'] \& {\C \x \C} \arrow[l,"\counit \x \counit"{pos=0.08, inner sep=1ex}, Rightarrow, shorten <= 1ex, shorten >= 19ex]\arrow[d,"\Gamma \x \Gamma"]  \\
    {} \& {} \& {(\T \x \F)\x (\T \x \F)} \arrow[ld,"\pi_{1,4}"] \\ 
    {} \& {\T \x \F} \arrow[r,"\Gamma'"'] \& {\C}
\end{cd}}
\end{eqD*}

So given $X_1,X_2,X_3,X_4\in \C$, $Q_{\mu}$ has component
$$(Q_{\mu})_{X_1,X_2,X_3,X_4}\: \Gamma'(T^{\Gamma'(T^{X_1}, F^{X_2})},F^{\Gamma'(T^{X_3}, F^{X_4})})\aar{\Gamma'(\counit^1_{T^{X_1},F^{X_2}}, \counit^2_{T^{X_3},F^{X_4}})} \Gamma' (T^{X_1}, F^{X_4}).$$
Moreover, we define $Q_\eta\:  \id{\C}\Rightarrow Q \circ \eta_{\C}$ as the pasting
\begin{cd}
{\C} \arrow[rrdd,"", equal, bend right=40] \arrow[r,"\eta_C"] \arrow[rd,"\Gamma"'] \& {\C \x \C} \arrow[r,"\Gamma \x \Gamma"] \& {(\T \x \F) \x (\T \x \F)} \arrow[d,"\pi_{1,4}"]\\[-3ex]
{} \& {\T \x \F} \arrow[ru,"\eta_{\T \x \F}"{description}] \arrow[r,"", equal] \& {\T \x \F} \arrow[d,"\Gamma'"]\\[-1.5ex]
{} \& {} \arrow[ru,"\unit{}"', Rightarrow, shorten <= 7ex, shorten >= 5ex] \& {\C.}
\end{cd}
So, given $x\in\C$, $Q_{\eta}$ has component
$$(Q_{\eta})_X\: X \aar{\unit_X} \Gamma'(\Gamma(X))=\Gamma'(T^X, F^X).$$
It is then straightforward to check that $Q_\mu$ and $Q_\eta$ satisfy the required coherence conditions using the naturality of $\counit$ and the triangular identities of the adjunction $\Gamma\dashv\Gamma'$. Hence, $(\C,Q, Q_\mu, Q_\eta)$ is a pseudo-algebra for $M$.

  We now prove that every morphism $G\colon (\mathcal{C},\mathcal{T},\mathcal{F})\to (\mathcal{D},\mathcal{T}',\mathcal{F}')$ between symmetrical rectangular pretorsion theories gives a pseudo-morphism between pseudo-algebras. Let $\Gamma \colon \C \to \T \times \F$ and $\Delta \colon \D \to \T' \times \F'$ be the associated adjoint equivalences, and $(\C, Q\colon \C \times \C \to \C, Q_{\mu}, Q_{\eta})$ and $(\D, R\colon \D \times \D \to \D, R_{\mu}, R_{\eta})$ be the pseudo-algebras corresponding to $(\mathcal{C},\mathcal{T},\mathcal{F})$ and $(\mathcal{D},\mathcal{T}',\mathcal{F}')$, defined as above. We construct a natural transformation
\sq[i][6][6][\phi][3][2.8]{\C\x \C}{\D\x\D}{\C}{\D}{G\x G}{Q}{R}{G}
such that $(G,\phi)$ is a pseudo-morphism between pseudo-algebras. We define $\phi$ as the following pasting:
\begin{cd}
    {\C \x \C} \arrow[rrr,"G\x G"] \arrow[d,"\Gamma \x \Gamma"'] \&[-2ex] {} \&[-2ex] {} \&[-2ex] {\D \x \D} \arrow[d,"\Delta \x \Delta"] \arrow[llld, twoiso, "\lambda \x \lambda", shorten <= 12ex, shorten >= 14ex] \\[-0.5ex]
    {(\T\x \F)\x (\T \x \F)} \arrow[rrr,"(G\x G) \x (G\x G)"] \arrow[d,"\pi_{1,4}"'] \&[-2ex] {} \&[-2ex] {} \&[-2ex] {(\T'\x \F')\x (\T' \x \F')} \arrow[d,"\pi_{1,4}"]\\[-2.85ex]
    {\T\x \F} \arrow[ddd,"\Gamma'"'] \arrow[rd, equal]\arrow[rrr, "G\x G"] \&[-2ex] {} \&[-2ex] {} \&[-2ex] {\T' \x \F'} \arrow[ldd,"\unit^{-1}",twoiso, shift left=4ex, shorten <= 5ex, shorten >= 5ex]\arrow[ddd,"\Delta'"]\\[-3ex]
    {} \&[-2ex] {\T\x \F} \arrow[rd, Rightarrow, "\lambda^{-1}", shorten <= 1ex, shorten >= 1ex] \arrow[ldd, Rightarrow, "\counit^{-1}", shorten <= 5ex,shorten >= 5ex,shift left=-6ex] \arrow[rru, "G\x G"]  {} \&[-2ex] {}\\[-3ex]
    {} \&[-2ex] {} \&[-2ex] {\D} \arrow[ruu, "\Delta"] \arrow[rd,equal,]\&[-2ex] {}\\[-3ex]
    {\C} \arrow[rrr,"G"']  \arrow[ruu,"\Gamma"'] \arrow[rru,"G"] \&[-2ex] {} \&[-2ex] {} \&[-2ex] {\D,}
\end{cd}
where $\lambda\: \Delta \c G \to (G\x G) \c \Gamma$ is the natural transformation associated to $G$ defined in \conx\ref{conslambda}. Such $\phi$ is natural by construction. Moreover, $(G,\phi)$ satisfies the axioms of pseudo-morphism of pseudo-algebras thanks to the naturality of the natural transformations involved and the triangular identities of the adjunctions. The assignment $G \mapsto (G,\phi)$ is functorial. Indeed, if $G=\id{\C}$ then by uniqueness of $\lambda^1_X$ and $\lambda^2_X$ for every $X\in\C$, we obtain $\lambda=\id{}$ and hence $\phi=\id{}$. Moreover, it is straightforward to prove that the assignment preserves composition, using again the uniqueness of $\lambda_X$ for every $X\in \C$ and the triangular identities of the adjunctions involved. 

We now show that every $2$-cell $\alpha\:G \Rightarrow H\:(\C,\T,\F)\to (\D,\T',\F')$ in $\SCartPTors$ gives a 2-cell between pseudo-morphism of pseudo-algebras. It suffices to prove that the natural transformation $\alpha$ satisfies the required axiom for a 2-cell between pseudo-morphisms. This follows then from the key equality of \prox \ref{propalphacomp} applied to both $\alpha$ and $\alpha \times \alpha$, together with the interchange law of pasting diagrams. Moreover, this assignment on 2-cells is trivially functorial. We have thus defined a 2-functor $$\chi\:\SCartPTors \to \Alg+{M}$$
over $\Biqpointcat$.

We prove that every pseudo-algebra $(\C,Q,Q_{\mu}, Q_{\eta})$ gives a symmetrical rectangular pretorsion theory, of the form $(\C,\T,\F)$. We define $\T$ as the full subcategory on the objects of the essential image of $\C \x \0 \subseteq \C \x \C \aar{Q} \C$. Analogously, we define $\F$ using $\1 \x \C \subseteq \C \x \C \aar{Q} \C$. We then define a functor $Q'\: \C \to \T \x \F$, as the composite
$$\C \aar{\eta_{\C}} \C \x \C \aar{\Gamma} (\C \x \0) \x (\1\x \C) \aar{Q \x Q} \T \x \F.$$
Here, $\Gamma$ is induced by the inclusions $\C\to \C\x \0$ and $\C\to \1\x \C$; it will coincide with the canonical equivalence associated to the free algebra $\C\x\C$. So given $X\in \C$, we have $Q'(X)=(Q(X,0), Q(1,X))$. We show that $Q'$ is an equivalence of categories with pseudo-inverse given by the composite $\T \x \F \subseteq \C \x \C \aar{Q} \C$. We construct the natural isomorphism $\gamma\:Q\c \operatorname{inc}\c \hspace{0.1ex}Q'\aR{}\Id{\C}$, where $\operatorname{inc}\:\T\x\F\cont \C\x\C$, as the following pasting
\begin{cd}[4.5]
    {\C} \arrow[r,"\eta_{\C}"] \arrow[rrrddd,"", equal, bend right=30] \& {\C \x \C} \arrow[r,"\Gamma"] \arrow[rdd,"",equal, bend right=20] \& {(\C \x \0)\x (\1\x \C)}  \arrow[d,"",hook] \arrow[r,"Q \x Q"] \& {\T \x \F} \arrow[d,"",hook]\\[-2.5ex]
    {} \& {} \& {(\C \x \C)\x (\C \x \C)} \arrow[r,"Q \x Q"] \arrow[d,"\pi_{1,4}"]\& {\C \x \C} \arrow[dd,"Q"] \arrow[ld,"Q_{\mu}",Rightarrow, shorten <=4ex, shorten >=4ex]\\[-1.5ex]
    {} \& {} \& {\C \x \C} \arrow[ld,"Q^{-1}_{\eta}"{pos=0.1}, Rightarrow, shorten <= 0.1ex, shorten >=13ex] \arrow[rd,"Q"] \& {}\\[-2.5ex]
    {} \& {} \& {} \& {\C.}
\end{cd}
Notice that the existence of the natural isomorphism $\gamma$ implies that $Q$ is full.
We then construct the the natural isomorphism $\gamma'\:Q'\c Q\c \operatorname{inc}\aR{}\Id{\T\x\F}$ to have component $\gamma'_{(T,F)}$ on $(T,F)\in \T \x \F$ given by the composite
$$(Q(Q(T,F),0), Q(1,Q(T,F)))\iso (Q(T,0),Q(1,F))\iso (T,F),$$
where the first isomorphism is given by $Q_{\mu}$ while the second one is given on components as follows. Given $T\in \T$, there exists $X\in \C$ such that $T\iso Q(X,0)$ and we have the following chain of isomorphisms
\begin{equation}\label{eqisoT}
Q(T,0)\iso Q(Q(X,0),0) \aiso{Q_\mu} Q(X,0)\iso T.
\end{equation}
Analogously, given $F\in \F$ there exists $Y\in \C$ such that $F\iso Q(1,Y)$ and we consider
\begin{equation}\label{eqisoF}
Q(1,F)\iso Q(1,Q(1,Y)) \aiso{Q_\mu} Q(1,Y)\iso F.
\end{equation}
Moreover, these isomorphisms do not depend on the choice of the objects $X$ and $Y$, respectively. Indeed, suppose that there exist $X,X'\in \C$ together with isomorphisms $j\:T \to Q(X,0)$ and $j'\: T \to Q(X',0)$. Then, since $Q$ is full, there exists a morphism $X \aar{k} X'$ such that the composite $Q(X,0) \aar{j^{-1}} T \aar{j'} Q(X',0)$ is equal to $Q(k,\id{})$. And the naturality of $Q_\mu$ yields the following commutative diagram
\begin{cd}[4.5]
   {} \&[-3ex] {Q(Q(X,0),0)}  \arrow[dd,"{Q(Q(k,\id{}),\id{})}"{description}]\arrow[r,"{(Q_{\mu})_{X,0,0,0}}"{inner sep=1.5ex},aiso] \& {Q(X,0)} \arrow[dd,"{Q(k,\id{})}"{description}] \arrow[rd,"{j^{-1}}"{inner sep=1ex},aiso] \&[-3ex] {} \\[-2ex]
   {Q(T,0)} \arrow[rd,"{Q(j',\id{})}"'{inner sep=1ex},aiso]\arrow[ru,"{Q(j,\id{})^{-1}}"{inner sep=0.5ex},aiso] \& {} \& {} \& {T} \\[-2ex]
   {} \& {Q(Q(X',0),0)} \arrow[r,"{(Q_{\mu})_{X',0,0,0}}"'{inner sep=0.5ex},aiso] \& {Q(X',0)} \arrow[ru,"j"'{inner sep=0.5ex},aiso] \& {} 
\end{cd}
Analogously for the chosen isomorphism between $Q(1,F)$ and $F$. It is then straightforward to prove that $\gamma'$ is 2-natural using that $Q$ is full and the naturality of $Q_{\mu}$. So $Q'\: \C \to \T \x \F$ is an equivalence of categories with pseudo-inverse given by the composite $\T \x \F \subseteq \C \x \C \aar{Q} \C$. 

\noindent Notice that the equivalence $Q'$ preserves the initial object and so $Q'(0)=(0,Z)$ is the initial object in $\T \x \F$. This implies that $0$ is the initial object in $\T$ and $Z$ is the initial object in $\F$. Clearly, every morphism $0\aar{a} T$ in $\T$ is a monomorphism by assumption on $\C$. But also every morphism $Z \aar{a} F$ in $\F$ is a monomorphism. Indeed, $(\id{},a)$ is a monomorphism because $Q'$ is full and so there exists a monomorphism $0\aar{v} Q(0,F)$ such that the monomorphism $Q'(v)$ is the composite 
$$Q'(0) \iso Q'(Q(0,Z)) \iso (0,Z) \aar{(\id{},a)} (0,F) \iso Q'(Q(0,F)).$$
This implies that $a$ is a monomorphism. Analogously, since $Q'(1)=(Z,1)$, we have that $Z$ is terminal in $\T$ and $1$ is terminal in $\F$. Morevoer, every morphism $F\to 1$ in $\F$ is an epimorphism and every morphism $T \to Z$ in $\T$ is an epimorphism. We conclude that both $\T$ and $\F$ are bi-quasi-pointed categories. As a consequence, $\T \x \F$ has a symmetrical rectangular pretorsion theory given by the product of $(\T, \T,Z)$ and $(\F,Z,\F)$. Indeed, $Z$ is terminal in $\T$ and initial in $\F$. By the proof \corx\ref{corolltransfer}, the symmetrical rectangular pretorsion theory on $\C$ transferred via the equivalence $Q'$ is exactly $(\C,\T,\F)$. Again by the proof of \corx\ref{corolltransfer}, it is easy to see that, given $X\in \C$, the following is a short exact sequence:
$$Q(X,0)\aar{Q(\id{},!)} Q(X,X) \aiso{Q^{-1}_{\eta}} X \aiso{Q_{\eta}} Q(X,X) \aar{Q(!,\id{})} Q(1,X).$$
We choose this as the short exact sequence associated to $X$. So that the canonical equivalence $\C\to \T\x\F$ is exactly $Q'$.

We now prove that every pseudo-morphism $(G,\phi)\: (\C, Q\colon \C \times \C \to \C, Q_{\mu}, Q_{\eta}) \to (\D, R\colon \D \times \D \to \D, R_{\mu}, R_{\eta})$ between pseudo-algebras gives a morphism between symmetrical rectangular pretorsion theories. Let $(\mathcal{C},\mathcal{T},\mathcal{F})$ and $(\mathcal{D},\mathcal{T}',\mathcal{F}')$ be the symmetrical rectangular pretorsion theories associated to
$(\C, Q, Q_{\mu}, Q_{\eta})$ and to $(\D, R, R_{\mu}, R_{\eta})$ respectively, with corresponding equivalences $Q'\: \C \to \T \times \F$ and $R'\: \D \to \T' \times \F'$. We prove that $G\: \C \to \D$ is a morphism of symmetrical rectangular pretorsion theories. Since $(G,\phi)$ is a pseudo-morphism of pseudo-algebras, $G$ preserves $0$ and $1$. Given $T\in \T$, there exists $X\in \C$ such that $T\cong Q(X,0)$. So
\v[-3]
$$G(T)\iso G(Q(X,0)) \aiso{\phi_{x,0}^{-1}} R(G(X),0).$$ But $R(G(X),0)\in \T'$ and $\T'$ is replete, so $G(T)\in \T'$. Analogously, given $F\in \F$ there exists $X\in \C$ such that $F\iso Q(1,X)$ and hence $G(F)\in \F'$. It remains to prove that $G$ sends the chosen short exact sequences in $(\mathcal{C},\mathcal{T},\mathcal{F})$ to short exact sequences in $\D$. $G$ sends the chosen short exact sequence for $X\in \C$ to the top row of the following diagram
\[\begin{tikzcd}
	{G(Q(X,0))} &&[-4.5ex] {G(Q(X,X))} &[-3ex] {G(X)} &[-3ex] {G(Q(X,X))} &&[-4.5ex] {G(Q(1,X))} \\
	{R(G(X,0))} && {R(G(X),G(X))} & {G(X)} & {R(G(X),G(X))} && {R(1,G(X)).}
	\arrow["{G(Q(\operatorname{id},!))}", from=1-1, to=1-3]
	\arrow["{\phi^{-1}_{X,0}}"', from=1-1, to=2-1,aiso]
	\arrow["{G((Q_{\eta}^{-1})_X)}", from=1-3, to=1-4,iso]
	\arrow[from=1-3, to=2-3,aiso,"\phi^{-1}_{X,X}"']
	\arrow["{G((Q_{\eta})_X)}", from=1-4, to=1-5, iso]
	\arrow[from=1-4, to=2-4,equal]
	\arrow["{G(Q(!,\operatorname{id}))}", from=1-5, to=1-7]
	\arrow[from=1-5, to=2-5,aiso,"\phi^{-1}_{X,X}"']
	\arrow[from=1-7, to=2-7,aiso, "\phi^{-1}_{1,X}"']
	\arrow["{R(\operatorname{id},!)}"', from=2-1, to=2-3]
	\arrow["{(R^{-1}_{\eta})_{G(X)}}"',iso, from=2-3, to=2-4]
	\arrow["{(R_{\eta})_{G(X)}}"', iso,from=2-4, to=2-5]
	\arrow["{R(!,\operatorname{id})}"', from=2-5, to=2-7]
\end{tikzcd}\]
Since the diagram is commutative (by naturality of $\phi$ and by pseudo-morphism axioms satisfied by $(G,\phi)$) and the second row is the chosen short exact sequence for $G(X)$ in $(\D,\F',\T')$, the first row is a short exact sequence. We conclude that $G$ is a morphism between symmetrical rectangular pretorsion theories. And the assignment $(G,\phi) \mapsto G$ is clearly functorial. Moreover, every 2-cell between pseudo-morphisms of pseudo-algebras trivially gives a 2-cell between morphisms of symmetrical rectangular pretorsion theories, that is simply its underlying natural transformation. And of course this assignment on 2-cells is functorial. We have thus defined a 2-functor $$\xi\:\Alg+{M} \to \SCartPTors$$
over $\Biqpointcat$.

We now show that the functors $\chi$ and $\xi$ give a 2-equivalence between the 2-categories $\SCartPTors$ and $\Alg+{M}$. We first prove that there exists a 2-natural isomorphism $\xi \circ \chi \Rightarrow \Id{}$. A symmetrical rectangular pretorsion theory $(\C,\T, \F)$ (with associated equivalence $\Gamma\: \C \to \T \x \F$) is sent by $\xi \circ \chi$ to the symmetrical rectangular pretorsion theory $(\C, \T', \F')$ where $\T'$ is the full subcategory on the objects of the essential image of $\C\x \0 \subseteq \C \x \C \aar{Q} \C$ and analogously $\F'$ is the full subcategory on the objects of the essential image of $\1 \x \C \subseteq \C \x \C \aar{Q} \C$. The associated equivalence $\Gamma': \C \to \T' \x \F'$ is given by
$$\C \aar{\eta_{\C}} \C \x\C \aar{\Gamma} (\C \x 0)\x (1 \x \C) \aar{Q \x Q} \T' \x \F',$$
where $Q\: \C \x \C \to \C$ is given by the definition of $\chi$ as above. So $\Gamma'$ sends $X\in \C$ to the pair $(Q(X,0),Q(1,X))\in \T'\x \F'$. Notice now that, given $T\in \T$, we have 
$$Q(T,0)= \Gamma'(T,Z) \iso \Gamma'(\Gamma(T))\iso T,$$
(with $Z\in \C$ the unique null object up to isomorphism) and so, since $\T'$ is replete, we conclude that $T\in \T'$. On the other hand, given $T'\in \T'$, there exists $X\in \C$ such that 
$$T'\iso Q(X,0)= \Gamma'(T^X,Z)\iso \Gamma'(\Gamma (T^X))\iso T^X$$
and so, since $\T$ is replete, we conclude that $T'\in \T$. We thus showed that $\T'=\T$ and completely analogously we have $\F'=\F$. So the identity functor $\Id{\C}\: \C \to \C$ is a morphism between $(\C,\T, \F)$ and its image $(\C,\T', \F')$ under $\xi \circ \chi$. And then $\id{}$ clearly gives a 2-natural isomorphism $\xi \circ \chi \Rightarrow \Id{}$ over $\Biqpointcat$.

\noindent We now exhibit a 2-natural isomorphism $\chi \circ \xi \Rightarrow \Id{}$. A pseudo-algebra $(\C, Q\: \C \x \C \to \C, Q_{\mu}, Q_{\eta})$ is sent by $\chi \circ \xi$ to the pseudo-algebra $(\C, R\: \C \x \C \to \C, R_{\mu}, R_{\eta})$, where $R\: \C \x \C \to \C$ is given by
$$\C \x \C \aar{Q' \x Q'} (\T \x \F) \x (\T \x \F) \aar{\pi_{1,4}} \T \x \F \subseteq \C \x \C \aar{Q} \C,$$
with $\T$, $\F$ and $Q'$ given by construction of $\xi$ and $\chi$ as above. And the corresponding $R_{\mu}$ and $R_{\eta}$ are defined by construction of $\xi$. We  construct an invertible pseudo-morphism of pseudo-algebras $(\Id{}, \phi)\:(\C, Q\: \C \x \C \to \C, Q_{\mu}, Q_{\eta})\to (\C, R\: \C \x \C \to \C, R_{\mu}, R_{\eta})$. To do so, we need to define the natural transformation 
\sq[i][6][6][\phi][3][2.8]{\C\x \C}{\C\x\C}{\C}{\C}{\Id{}\x \Id{}}{Q}{R}{\Id{}}
We define $\phi$ as the following pasting
\begin{eqD*}
\scalebox{0.8}{
\begin{cd}*
    {\C \x \C} \arrow[r,"\eta_{\C} \x \eta_{\C}"] \arrow[rd,"",equal]\&[-3ex] {(\C \x \C) \x (\C \x \C)} \arrow[r,"\Gamma \x \Gamma"] \arrow[d,"\pi_{1,4}"] \&[-3ex]  {(\C\x 0)\x (\C \x 1) \x (\C\x 0)\x (\C \x 1) } \arrow[r,"(Q\x Q)\x (Q\x Q)"{inner sep=1.2ex}]  \arrow[d,"\pi_{1,4}"]\&[-3ex] {(\T \x \F) \x (\T \x \F)} \arrow[d,"\pi_{1,4}"] \\
    {} \& {\C \x \C} \arrow[r,"\Gamma"] \arrow[rdd,"", equal] \& {(\C \x 0)\x (1 \x \C) } \arrow[r,"Q \x Q"] \arrow[d,"", hook] \& {\T \x \F} \arrow[d,"", hook] \\[-3ex]
    {} \& {} \& {(\C \x \C) \x (\C \x \C)} \arrow[r,"Q \x Q"] \arrow[d,"\pi_{1,4}"] \& {\C \x \C} \arrow[ld,"Q_\mu", twoiso, shorten <= 10 ex, shorten >= 10 ex ] \arrow[d,"Q"] \\
    {} \& {} \& {\C \x \C} \arrow[r,"Q"'] \& {\C.} 
\end{cd}}
\end{eqD*}
It is then straightforward to prove that $(\Id{},\phi)$ satisfies the coherence axioms required for a pseudo-morphism of pseudo-algebras, thanks to the fact the the isomorphisms \refs{eqisoT} and \refs{eqisoF} do not depend on choices of $X$ and $Y$. Analogously, the pair $(\Id{}, \phi^{-1})$ gives a pseudo-morphism of pseudo-algebras, which is the inverse of $(\Id{}, \phi)$. One can then easily show that the pairs $(\Id{},\phi)$ constructed as above are the components of a 2-natural isomorphism $\chi \circ \xi \Rightarrow \Id{}$ over $\Biqpointcat$. We thus conclude that $\SCartPTors$ is 2-equivalent to $\Alg+{M}$ over $\Biqpointcat$.
\end{proof}

%\begin{remark}
    %£££ rectangular bands
%\end{remark}

We show that \thex\ref{teormonad} restricts well to the pointed case. We denote by $\Pointcat$ the full sub-2-category of $\Biqpointcat$ of pointed categories.

\begin{theorem}\label{teormonadrestricted}
    The 2-category $\CartTors$ of rectangular torsion theories is 2-equivalent, over $\Pointcat$, to the 2-category $\Alg+{M^\ast}$ of pseudo-algebras (and pseudo-morphisms between them) of the (strict) 2-monad
    \begin{fun}
	M^\ast & \: & \Pointcat & \too & \Pointcat \\[1ex]
    && \tcv*{\C}{\D}{G}{H}{\alpha} & \mto & \tcv*{\C\x\C}{\D\x\D}{G\x G}{H\x H}{\alpha\x \alpha}
\end{fun}
\end{theorem}
\begin{proof}
    It is easy to see that the $2$-monad $M\:\Biqpointcat\to \Biqpointcat$ of \thex\ref{teormonad} restricts to the $2$-monad $M^\ast$ of the statement. It is also straightforward to see that the proof of \thex\ref{teormonad} restricts to a proof of this theorem. Here, $0$, $1$ and $Z$ all coincide with the zero object. Moreover, any morphism between rectangular torsion theories preserves the zero object, by \prox\ref{propcatcarttors}. Rather than using the definition of symmetrical rectangular pretorsion theory, we apply \prox\ref{proppointedcase}.
\end{proof}

\begin{corollary}
    The forgetful $2$-functor $\SCartPTors\to \Biqpointcat$ has a left bi-adjoint, (essentially) given by $M$.
    The forgetful $2$-functor $\CartTors\to \Pointcat$ has a left bi-adjoint, (essentially) given by $M^\ast$.
\end{corollary}

\begin{proof}
    By the theory of $2$-monads, the forgetful 2-functor $\Alg+{M}\to \Biqpointcat$ has a left bi-adjoint, given by sending $\C\in \Biqpointcat$ to the free $M$-algebra on $\C$, which coincides with calculating $M(\C)$ and endowing it with the algebra structure map $Q=\mu_\C\:M^2(\C)\to M(C)$. By \thex\ref{teormonad}, which shows a $2$-equivalence over $\Biqpointcat$, we conclude that the forgetful $2$-functor $\SCartPTors\to \Biqpointcat$ has the same left bi-adjoint, up to postcomposing with the 2-equivalence $\xi\:\Alg+{M}\approx \SCartPTors$.

    Analogously for the forgetful $2$-functor $\CartTors\to \Pointcat$, thanks to \thex\ref{teormonadrestricted}.
\end{proof}

We can also apply our monadicity result to construct (2-dimensional) limits in the 2-categories of symmetrical rectangular pretorsion theories and that of rectangular torsion theories.

\begin{corollary}
    The forgetful $2$-functor $\SCartPTors\to \Biqpointcat$ creates all bilimits.
    The forgetful $2$-functor $\CartTors\to \Pointcat$ creates all bilimits.
\end{corollary}
\begin{proof}
    It is known that (pseudo-)monadic 2-functors create bilimits, see \cite[Theorem 6.3.1.6]{O21}.
\end{proof}

\begin{remark}\label{remlimitsinpoincat}
    It is folklore that one can always formally adjoin a zero object to any category. More precisely, $\Pointcat$ is 2-equivalent to the 2-category of strict algebras and pseudo-morphisms for the $2$-monad $K$ on $\mathbf{Cat}$ that sends a category $\C$ to the category constructed by formally adjoining a zero object to $\C$. As a consequence,
    %(see \cite{}),
    %£££ Blackwell, Power, Kelly 2-monad theory
    $\Pointcat$ has all PIE limits, and in particular all bilimits, created by the forgetful $2$-functor into $\mathbf{Cat}$.
\end{remark}

\begin{corollary}\label{corlimitscarttors}
    $\CartTors$ has all bilimits, created by the forgetful $2$-functor $\CartTors\to \Pointcat$ and thus calculated in $\mathbf{Cat}$ (thanks to \remx\ref{remlimitsinpoincat}).
\end{corollary}

\begin{remark}
    The recipe for constructing weak products (also called biproducts) of rectangular torsion theories given by \corx\ref{corlimitscarttors} coincides with the recipe for products described in \thex\ref{TheB}. Exactly as the recipe for weak products also works in the setting of general pretorsion theories and gives products, one could try to take the recipes for bilimits given by \corx\ref{corlimitscarttors} and see if they also work in the general setting.
\end{remark}

\section{A case study: torsion theory classes of epimorphisms}

Let $\mathbb{X}$ be a pointed category. Consider any full subcategory $\mathcal{E}$ of the category of morphisms in $\mathbb{X}$ which contains all isomorphisms and all morphisms to a zero object. Assume $\mathcal{E}$ is such that each of its objects is an epimorphism in $\mathbb{X}$. Consider two full replete subcategories of $\mathcal{E}$: 
\begin{itemize}
\item $\mathcal{T}$, given by isomorphisms 
\item and $\mathcal{F}$, given by morphisms to a zero object. 
\end{itemize}
Note that $\mathcal{E}$ is a pointed category where zero objects are given by the intersection $\mathcal{T}\cap \mathcal{F}$. We write $0$ for a zero object in $\mathbb{X}$.

\begin{lemma}
For an object $e\colon X\to Y$ in $\mathcal{E}$, a diagram
$$\xymatrix{A\ar[r]^-{a}\ar[d]_-{} & X\ar[r]^-{be}\ar[d]^-{e} & B\ar[d]^-{\id{B}} \\ 0\ar[r] & Y\ar[r]_{b} & B }$$
is an exact sequence in $\mathcal{E}$ if and only if $a$ is a kernel of $e$ while $e$ is a cokernel of $a$, and, $b$ is an isomorphism.
\end{lemma}

From the lemma above we easily get the following.

\begin{theorem}\label{TheD}
$(\mathcal{E},\mathcal{T},\mathcal{F})$ is a torsion
theory if and only if every object of $\mathcal{E}$ is a cokernel in $\mathbb{X}$ and has a kernel in $\mathbb{X}$. When $\mathbb{X}$ has binary products, this torsion theory is rectangular if and only if every object of $\mathcal{E}$ is a product projection in $\mathbb{X}$. 
\end{theorem}

Let us say that $\mathcal{E}$ is a \emph{torsion theory class} of epimorphisms when $(\mathcal{E},\mathcal{T},\mathcal{F})$ is a torsion theory, and a \emph{rectangular torsion theory class} when this torsion theory is rectangular. The theorem above leads to the following characterization theorems.

\begin{theorem}\label{TheE}
The class of product projections in a pointed category having binary products is a torsion theory class if and only if the category has normal projections, and, if and only if it is a rectangular torsion theory class. 
\end{theorem}

\begin{theorem}\label{TheF}
The class of split epimorphisms in a pointed regular category is a torsion theory class if and only if the category is normal, and, if and only if the class of regular epimorphisms is a torsion theory class.
\end{theorem}

In the dual of the category of pointed sets every regular epimorphism is a product projection, while in any abelian category every split epimorphism is a product projection.

\begin{theorem}\label{TheG}
The class of split epimorphisms in a pointed regular category is a rectangular torsion theory class if and only if every split epimorphism is a product projection. The class of regular epimorphisms is a rectangular torsion theory class if and only if every regular epimorphism is a product projection.
\end{theorem}

\end{document}